\documentclass{article}
\usepackage{graphicx,subcaption}
\usepackage{mathtools}
\mathtoolsset{showonlyrefs}
\usepackage{amssymb}
\usepackage{amsmath}
\usepackage{amsthm}
\usepackage{xcolor,soul}
\usepackage{mdframed}
\usepackage{fancyhdr}
\usepackage{epstopdf}
\usepackage[colorlinks=true,citecolor=magenta,linkcolor=blue]{hyperref}
\newcommand{\ve}{\varepsilon}

\makeatletter
\let\@fnsymbol\@arabic
\makeatother

\newcommand{\dd}{\mathrm{d}}

\usepackage[]{geometry}
\definecolor{greenEE}{RGB}{0, 128, 0}

\newtheorem{theorem}{Theorem}
\newtheorem{lemma}{Lemma}
\newtheorem{proposition}{Proposition}
\newtheorem{remark}{Remark}

\usepackage{tikz}
\usetikzlibrary{decorations.pathreplacing}
\usetikzlibrary{positioning}
\usetikzlibrary{arrows}
\usetikzlibrary{arrows.meta}
\usetikzlibrary{calc}
\usetikzlibrary{shapes}

\graphicspath{ {./figures/} }

\title{Slow-fast dynamics in a planar parasite--host model with an extinction singularity}
\author{Jacopo Borsotti\thanks{School of Resource and Environmental Management, Simon Fraser University, 8888 University Drive, V5A 1S6 Burnaby, Canada, {\tt jacopo\_borsotti@sfu.ca}}\,, Hildeberto Jardón-Kojakhmetov\thanks{Johann Bernoulli Institute for Mathematics and Computer Science, University of Groningen, P.O. Box 407, 9700 AK, Groningen, The
Netherlands, {\tt h.jardon.kojakhmetov@rug.nl}}, and Mattia Sensi\thanks{Dipartimento di Matematica, Università degli Studi di Trento, Via Sommarive 14, 38123 Povo (Trento), Italy, {\tt mattia.sensi@unitn.it}}}
\date{}

\begin{document}

\maketitle

\begin{abstract}
    We study a slow-fast parasite--host model featuring a singularity at the extinction state. Using techniques from Geometric Singular Perturbation Theory (GSPT), and in particular the so-called blow-up method, we desingularize that point and reconstruct the local and global dynamics. The system we consider is in non-standard GSPT form and is characterized by a rich dynamical behavior: families of slow-fast homoclinic orbits, canard-like transitions generated by trajectories that remain close to a repelling critical manifold, and topological changes produced by infinitesimal variations of the infection rate, including the creation and destruction of an endemic equilibrium. We also show that our model is able to reproduce dynamics observed in the spread of the Tasmanian devil facial tumor disease (DFTD), whose behavior resembles the one of a parasite. We conclude with a numerical exploration of the model, to illustrate our analytical results. 
\end{abstract}
\vspace{1em}
\\ 
\noindent \textbf{Dynamical systems arising in epidemiology and population biology can be characterized by ecological and infection processes evolving on separated timescales. In parasite--host interactions, this separation becomes particularly relevant when demographic processes occur much more slowly than infection and parasite-induced mortality. Such systems may develop singular structures associated with extinction states, where standard perturbation techniques fail and complex phenomena emerge. Motivated by these observations, we investigate a planar parasite-host model exhibiting a nonstandard slow-fast structure together with a singularity at host extinction. By exploiting Geometric Singular Perturbation Theory (GSPT) and the blow-up method, we characterize the local and global organization of the phase space and identify several remarkable behaviors, including slow-fast homoclinic orbits, canard-like transitions near repelling critical manifolds, and abrupt topological changes generated by infinitesimal parameter variations. Beyond its mathematical interest, the model provides a framework for understanding transmissible diseases capable of persisting at very low host density, such as the Tasmanian devil facial tumor disease (DFTD), whose frequency-dependent transmission dynamics have been associated with severe population declines and possible host extinction.} 

\section{Introduction}

Understanding how parasites interact with their hosts has been a central question in theoretical biology, for example it is now widely believed that parasites were responsible for several extinctions on islands \cite{MCCALLUM1995190}. Indeed, parasites can reduce host fecundity and increase death rates \cite{webster1999cost}. A natural mathematical instrument to study these biological dynamics are systems of ordinary differential equations (ODEs), possibly adapting the numerous models that have been developed over the years to simulate epidemics, starting from the pioneering work of Kermack and McKendrick \cite{555}. However, those models fail to explain host extinction phenomena, which are observed in some particular situations (see, e.g., \cite{lae,SCOTT1989176,collapse}). 

Parasite models are often formulated as SI (Susceptible -- Infected) compartmental models, with the underlying assumption that an individual carrying a parasite will carry it until death \cite{webster1999cost,mandal2011mathematical,kang2012multiscale,lorenzi2021evolutionary,gomez2024eco}. To understand the parasite induced host extinction, Ebert, Lipsitch, and Mangin \cite{ebert} formulated a planar ODE model, which was later refined by Hwang and Kuang \cite{hwang2003deterministic} and whose structure is
\begin{align} \label{eq:model_original}
\begin{aligned}
    \frac{\dd S(t)}{\dd t} &= b (S(t)+ \theta I(t))(M - S(t) - I(t)) - n S (t) - \beta \frac{S(t) I(t)}{S(t) + I(t)}, \\ 
    \frac{\dd I(t)}{\dd t} &= \beta \frac{S(t) I(t)}{S(t) + I(t)} - d I(t) - n I(t), \\ 
\end{aligned}
\end{align}
where: 
\begin{itemize}
    \item $b > 0$ represents the maximum per capita birth rate of uninfected hosts; 
    \item $\theta \in [0,1]$ is the relative fecundity of an infected host; 
    \item $M>0$ represents the environmental carrying capacity;
    \item $\beta > 0$ is the infection rate; 
    \item $d>0$ is the parasite-induced death rate; 
    \item $n>0$ represents the parasite-independent death rate of both susceptible and infected hosts.  
\end{itemize}
The refinement consisted of replacing the nonlinear infection term $\beta S I$ with $\beta S I / (S+I)$. The reason behind this choice is that $\beta$ represents an infection rate, i.e., it is the maximum number of infections an infective host can cause in a unit of time, implying the necessity of normalizing it by the total population. Hwang and Kuang also remark that ``for large populations, individual’s finite and often slow movement prevents it to make contact to a large number of individuals in a unit time. Such a mechanism is better described by $\beta S I / (S+I)$ than $\beta S I$”. Indeed, the model was originally designed to reproduce the effect of microparasites on a population of \emph{Daphnia magna}, a planktonic crustacean.

In this work, we introduce an additional (slow) timescale in the Hwang and Kuang model, arising from the assumption that the birth rate and the parasite-independent death rate are small relative to the other quantities governing the model. Moreover, we show how our model is able to reproduce dynamics observed in the spread of the Tasmanian devil facial tumor disease (DFTD).

\subsection{Case study: Tasmanian devil facial tumor disease} \label{sec:DFTD}

The DFTD is a highly aggressive transmissible cancer affecting Tasmanian devils \cite{what_is_it}. The disease is transmitted directly between individuals, mainly through biting during aggressive social interactions and mating behavior \cite{nature}. Infected animals develop large facial tumors that progressively impair feeding and typically lead to death within a few months \cite{survival}.

DFTD has attracted considerable attention because its transmission is considered frequency-dependent rather than density-dependent \cite{frequency}. Indeed, Tasmanian devils continue to interact aggressively even when population density becomes very low, allowing the disease to persist despite severe population declines. It is important to remark that frequency-dependent diseases do not burn out at low host density and can cause extinction \cite{mech_ext}. 

DFTD has caused dramatic reductions in devil populations across Tasmania, with declines up to 89 \% in some regions \cite{89}. Some studies even predicted the possible extinction of the host population \cite{frequency, may_extinction}, while others suggest that long-term coexistence between host and disease may occur through low-density endemic states and a transition from
iteroparity toward single breeding, associated with precocious sexual maturity \cite{history}. Indeed, due to the high disease mortality, infected individuals might not survive for a second breeding. 

As remarked in \cite{history}, the DFTD ``behaves'' like a parasite. This motivates the choice of model \eqref{eq:model_original} to reproduce the above observed phenomena. Clearly, it is also reasonable to assume that the birth rate and the DFTD-independent death rate are small relative to the other quantities governing the model. Indeed, Tasmanian devils live up to five years, while DFTD leads to death within a few months \cite{survival}. Moreover, the logistic growth assumption is biologically reasonable for colonial or social animal populations, where colony size is constrained by resource competition, nesting availability, and social stress \cite{brown}. Finally, note that the relative fecundity of infected devils $\theta$ can help reproduce the mentioned transition from iteroparity toward single breeding. 

\subsection{Outline of this work}

In order to analyze our model, we exploit techniques of Geometric Singular Perturbation Theory (GSPT), stemming from the seminal work of Neil Fenichel \cite{23}, and the more recent \emph{blow-up} technique \cite{40,survey}. 
Such tools have proven to be effective for analyzing population dynamics evolving on different timescales, such as predator-prey interactions \cite{allee, leslie} and epidemics \cite{SIRS, jardon2021geometric, kaklamanos2024geometric,jardon2021geometric2, bulai2024geometric}. 

The system we analyze is in \emph{non-standard form} \cite{wechselberger2020geometric} and presents a singularity at the extinction state, which we desingularize through the blow-up method. Depending on the relationship between the parameters governing the system, we highlight different remarkable behaviors, such as families of slow-fast homoclinic orbits and canard-like transitions generated by trajectories that approach a repelling critical manifold (see Figure \ref{fig:homoclinic} and Appendix \ref{sec:blow-up_d}); we showcase all the possible dynamics in Figure \ref{fig:sketch}, and devote most of the manuscript to rigorously prove under which conditions each of them occurs. We exploit the separation of time scales between infection dynamics and demography to fully characterize the asymptotic behavior of our system. Another 
noteworthy feature of the model, both from a mathematical and a biological perspective, is the existence of the endemic equilibrium only for values of the Basic Reproduction Number $\mathcal{R}_0$ very close to 1 (this will be made precise in Section \ref{section:equilibria}; see also Figures \ref{fig:bifurcation} and \ref{fig:end_eq_grad}): infinitesimal variations of the infection rate $\beta$ result in the creation and destruction of a globally stable endemic equilibrium.

We remark that homoclinic and heteroclinic orbits were observed and proven analytically to exist in \cite{li2012global}, a short blow-up analysis of the origin was carried out in \cite{berezovskaya2004simple}, and host extinction was studied in \cite{hwang2003deterministic}. However, the analysis we present in this work collects and complements all the results found in the literature, expanding them and proving them rigorously through the lenses of GSPT (as opposed to more classical approaches found in the aforementioned works) and a much more detailed blow-up exploration of the origin.

The remainder of this manuscript is structured as follows. In Section \ref{sec:model}, we introduce the model under study and we fully analyze its equilibria, nullclines, and bifurcations. In Section \ref{sec:mult_time}, we study the two-timescale nature of the system, also exploiting the blow-up results reported in Appendix \ref{section:blowup}. In Section \ref{sec:numerics}, we illustrate our main analytical results through several numerical simulations. We conclude in Section \ref{sec:concl}, where we also discuss the connections between our model and the spread of the DFTD. 

\section{The model} \label{sec:model}

In this section, we introduce a slow-fast version of an ODE model introduced in \cite{hwang2003deterministic} (and further analyzed in \cite{berezovskaya2004simple,li2012global}) to reproduce the parasite-induced host extinction observed in \cite{ebert}. 

Recall the original model \eqref{eq:model_original}. We assume that the birth rate $b$ of susceptible hosts $S$ and the parasite-independent death rate $n$ of both susceptible and infected $I$ hosts are much smaller than the other rates governing the system. These assumptions lead to the following slow-fast system in non-standard form \cite{wechselberger2020geometric}:
\begin{align} \label{eq:model}
\begin{aligned}
    \frac{\dd S(t)}{\dd t} &= \varepsilon a (S(t)+ \theta I(t))(M - S(t) - I(t)) - \varepsilon S (t) - \beta \frac{S(t) I(t)}{S(t) + I(t)}, \\ 
    \frac{\dd I(t)}{\dd t} &= \beta \frac{S(t) I(t)}{S(t) + I(t)} - d I(t) - \varepsilon I(t), \\ 
\end{aligned}
\end{align}
where $0 < \varepsilon \ll 1$ and $a>0$ represents the ratio between the birth and death rates of the susceptible individuals. In particular $\varepsilon \ll a,M,\beta,d$.

Define $N(t)=S(t)+I(t)$, which represents the total number of hosts. This quantity obeys 
\begin{equation}
    \frac{\dd N(t)}{\dd t} = \varepsilon a (S(t)+ \theta I(t))(M - N(t)) - \varepsilon N(t) - d I(t), 
\end{equation}
which simplifies to 
\begin{equation}
    \frac{\dd N(t)}{\dd t} = \varepsilon N(t)(a(M - N(t))-1)  
\end{equation}
if the parasites are not present ($I \equiv 0$). This leads to the concept of the basic demographic reproduction number $\mathcal{R}_d$ \cite{berezovskaya2004simple}, which
is given by 
\begin{equation} \label{eq:demographic}
    \mathcal{R}_d \coloneqq aM - 1,
\end{equation}
and determines whether the population can grow ($\mathcal{R}_d>0$) or it does not survive ($\mathcal{R}_d<0$) in absence of the parasite. When $\mathcal{R}_d>0$, linearization around the parasite-free equilibrium $(M-1/a, 0)$ of system \eqref{eq:model} allows us to determine the basic reproduction number $\mathcal{R}_0$, which can be derived through the Next Generation Matrix method \cite{basic} (in the case of system \eqref{eq:model}, this matrix is actually a scalar quantity, as there is only one infectious compartment/equation) and is given by 
\begin{equation} \label{eq:reproduction}
    \mathcal{R}_0 \coloneqq \frac{\beta}{d + \varepsilon}.
\end{equation}
Classically, $\mathcal{R}_0$ represents a threshold quantity which allows us to distinguish between scenarios in which the parasites will successfully invade the population ($\mathcal{R}_0>1$) and those in which the parasitic infection will die out ($\mathcal{R}_0<1$). We will show that, in the case of system \eqref{eq:model}, this distinction does not provide a complete understanding of the asymptotic dynamics (see in particular Proposition \ref{prop:EE}, Theorem \ref{teo:EE}, and Figure \ref{fig:sketch}).

Applying the normalizing change of variables 
\begin{equation*}
    S = M u, \quad I=M v,
\end{equation*}
and dropping the dependence on the time variable for ease of notation, system \eqref{eq:model} can be rewritten as
\begin{align} \label{eq:rescaled}
\begin{aligned}
    \dot{u} &= \varepsilon \alpha (u + \theta v)(1-u-v) - \varepsilon u - \beta \frac{u v}{u+v}, \\ 
    \dot{v} &= \beta \frac{u v}{u + v} - d v - \varepsilon v, \\ 
\end{aligned}
\end{align}
where $\alpha=a M > 0$.

Note that the basic demographic reproduction number changes accordingly, whereas the basic reproduction number remains the same:
\begin{equation} \label{eq:numbers}
    \mathcal{R}_d = \alpha - 1, \quad \mathcal{R}_0 = \frac{\beta}{d + \varepsilon}. 
\end{equation}
System \eqref{eq:rescaled} evolves in the biologically relevant region 
\begin{equation} \label{eq:Delta} 
\Delta \coloneqq \{(u,v)  \in \mathbb{R}^2 \colon \; u \ge 0, \; v \ge 0, \; u+v \le 1\} 
\end{equation}
since $\dot{v}|_{v=0}=0$, $\dot{u}|_{u=0}=\varepsilon a \theta v (1-v) \ge 0$, and 
$\dot{u}+\dot{v}|_{u+v=1}= - \varepsilon - d v < 0$. 

\subsection{Equilibria and bifurcation analysis} \label{section:equilibria}

We begin the analysis of system \eqref{eq:rescaled} with its equilibria. Regardless of the values attained by the parameters of the system, it admits a singular equilibrium point $\mathbf{x_0}=(0,0)$. Indeed, even though the equations are not well-defined at $\mathbf{x_0}$, we consider it a sort of equilibrium since we will show that the orbits converge toward it under some circumstances. Clearly, it is not possible to talk about the stability of $\mathbf{x_0}$ with standard linear algebra tools. 

If $\mathcal{R}_d > 0$ (i.e., if $\alpha > 1$; recall \eqref{eq:numbers}) there exists the disease-free equilibrium (DFE) $\mathbf{x_1}=(1-1/\alpha, 0)$. If $\mathcal{R}_d < 0$, then all individuals will die even in absence of parasites in the population ($(u,v) \to \mathbf{x_0}$), hence a DFE with positive susceptible population cannot exist. The eigenvalues of the Jacobian matrix of \eqref{eq:rescaled} computed at $\mathbf{x_1}$ are $\lambda_1=-\varepsilon \alpha(1-1/\alpha) < 0$ and $\lambda_2=\beta - d -\varepsilon$, hence the local stability of the DFE depends on the basic reproduction number $\mathcal{R}_0$ \eqref{eq:numbers}. In particular, if $\mathcal{R}_0 > 1$ then the DFE is unstable, while if $\mathcal{R}_0 < 1$ then it is locally asymptotically stable. 

The following proposition shows that, assuming $\varepsilon \ll 1$, an endemic equilibrium (EE; i.e., an equilibrium $\mathbf{x_2}=(u_2, v_2)$ such that $v_2>0$) exists under rather restrictive conditions. 

\begin{proposition} \label{prop:EE} 
Assume that $\alpha > 1$ and that $\beta = d + \varepsilon k$ with 
\begin{equation} \label{eq:k}
    1 < k < \frac{\alpha \left(1- \frac{\varepsilon\theta}{d + \varepsilon}\right)}{1-\alpha \frac{\varepsilon \theta}{d+ \varepsilon}} \eqqcolon \alpha^* = \alpha + \mathcal{O}(\varepsilon), 
\end{equation}
then there exists the EE $\mathbf{x_2}=(u_2, v_2)$ defined by 
\begin{equation} \label{eq:EE}
\begin{aligned}
    u_2 &= \frac{1}{\mathcal{R}_0} \left(1-\frac{\beta - d}{\varepsilon\alpha} \left(1-\theta + \frac{\theta \beta}{d + \varepsilon}\right)^{-1}\right), \\ 
    v_2 &= u_2 (\mathcal{R}_0 - 1) = \left(1 - \frac{1}{\mathcal{R}_0}\right) \left(1-\frac{\beta - d}{\varepsilon\alpha} \left(1-\theta + \frac{\theta \beta}{d + \varepsilon}\right)^{-1}\right).  
\end{aligned}
\end{equation}
\end{proposition}
\begin{proof}
    Imposing that $\dot{v}=0$ in \eqref{eq:rescaled} we immediately obtain $v_2 = u_2 (\mathcal{R}_0 - 1)$, which can be satisfied in the biologically feasible region if and only if $\mathcal{R}_0>1$, i.e. $\beta > d + \varepsilon$. Similarly, imposing $\dot{u}=0$ we obtain 
    \begin{equation} \label{eq:dim_teo1_bis}
        u_2 = \frac{d + \varepsilon}{\beta} \left(1-\frac{\beta - d}{\varepsilon\alpha} \left(1-\theta + \frac{\theta \beta}{d + \varepsilon}\right)^{-1}\right). 
    \end{equation}
    Since $\varepsilon \ll 1$, we need to have $\beta - d \in \mathcal{O}(\varepsilon)$ so that $u_2 > 0$. Combing the previous two requests obtain $\beta= d + \varepsilon k$ for some $k>1$ (since, by assumption, $\mathcal{R}_0>1$). Finally, it suffices to require that 
    \begin{equation} \label{eq:dim_teo1}
        1-\frac{\beta - d}{\varepsilon\alpha} \left(1-\theta + \frac{\theta \beta}{d + \varepsilon}\right)^{-1} = 1-\frac{k}{\alpha} \left(1 + \frac{\varepsilon\theta (k-1)}{d + \varepsilon}\right)^{-1} > 0, 
    \end{equation}
    which leads to \eqref{eq:k}. Indeed, note that both \eqref{eq:dim_teo1} and $(d+\varepsilon)/ \beta$ are smaller than $1$, implying that also $u_2 < 1$ from \eqref{eq:dim_teo1_bis}. 
\end{proof}

\begin{remark}
    For $\alpha>1$ and $0<\ve\theta\ll1$ sufficiently small, $\alpha<\alpha^*$. On the other hand, if $\theta=0$, then $\alpha=\alpha^*$. 
\end{remark}

Note that under the assumption of Proposition \ref{prop:EE}, since $\varepsilon \ll 1$, $\mathcal{R}_0 = 1 + \mathcal{O}(\varepsilon)$, implying that $u_2 = 1 - k/\alpha + \mathcal{O}(\varepsilon)$ and that $v_2 \in \mathcal{O}(\varepsilon)$. This means that the EE exists if the parasites have the ability to invade the population ($\mathcal{R}_0 > 1$) but only very slowly ($\mathcal{R}_0 \approx 1$). At the same time, also $\mathcal{R}_d > 0$ is a required condition. The following theorem describes the stability of the EE. 

\begin{theorem} \label{teo:EE}
    Under the hypothesis of Proposition \ref{prop:EE}, the EE $\mathbf{x_2}$ is locally asymptotically stable. 
\end{theorem}
\begin{proof}
Recall \eqref{eq:dim_teo1}; expanding \eqref{eq:EE} with respect to $\varepsilon$ we obtain
    \begin{equation} \label{eq:expansion1}
        u_2 = 1- \frac{k}{\alpha} + \varepsilon \frac{k-1}{\alpha d} (k+ \theta k - \alpha) + \mathcal{O}(\varepsilon^2)\in \mathcal{O}(1),
    \end{equation}
    while 
    \begin{equation} \label{eq:expansion2}
        v_2 = \varepsilon  \left(1 - \frac{k}{\alpha}\right) \frac{k-1}{d} + \mathcal{O}(\varepsilon^2) \in \mathcal{O}(\varepsilon). 
    \end{equation}
    The Jacobian of system \eqref{eq:rescaled} is given by
    \begin{equation}
        J = 
        \begin{bmatrix}
            \varepsilon \alpha (1-u-v) - \varepsilon \alpha (u+\theta v) - \varepsilon -\beta \frac{v^2}{(u+v)^2} & \varepsilon \alpha \theta (1-u-v)- \varepsilon\alpha (u+\theta v) - \varepsilon -\beta \frac{u^2}{(u+v)^2} \\ 
            \beta \frac{v^2}{(u+v)^2} & \beta \frac{u^2}{(u+v)^2} - d - \varepsilon\\ 
        \end{bmatrix}.
    \end{equation}
    Neglecting the $\mathcal{O}(\varepsilon^3)$-terms since $\varepsilon \ll 1$, a straightforward calculation shows that, evaluated at the EE $\mathbf{x_2}$, its characteristic polynomial is 
    \begin{equation}
        p(\lambda)=\lambda^2 - (\varepsilon (k-\alpha) + T)\lambda + \varepsilon^2 \frac{d \beta}{\alpha^2} (k-1)^2, 
    \end{equation}
    where $T \in \mathcal{O}(\varepsilon^2)$ incorporates several terms depending on $\varepsilon$, $\alpha$, $k$, and $d$. Therefore, the corresponding eigenvalues are 
    \begin{equation}
        \lambda_{1,2} = \frac{1}{2} \left(\varepsilon(k-\alpha)+T \pm \sqrt{(\varepsilon(k-\alpha)+T)^2 - \varepsilon^2\frac{4 d \beta}{\alpha^2}(k-1)^2}\right)
    \end{equation}
    and, since $k<\alpha$, their real parts are negative. Note the EE $\mathbf{x_2}$ can be a node or a focus, depending on the relationship between the parameters $k$, $\alpha$, and $d$. 
\end{proof}

For a bifurcation analysis regarding the equilibria $\mathbf{x_0}$, $\mathbf{x_1}$, and $\mathbf{x_2}$ we focus on the parameter $\beta$. Figure \ref{fig:bifurcation} shows the bifurcations diagrams (the leftmost depicts the $u$-coordinate of the equilibria, while the rightmost the $v$-coordinate) assuming $\alpha > 1$, otherwise neither $\mathbf{x_1}$ nor $\mathbf{x_2}$ would exist. As $\beta$ increases, there is a transcritical bifurcation for $\beta= d + \varepsilon$ (i.e., for $k=1$; recall \eqref{eq:k}) with $\mathbf{x_2}$ coinciding with $\mathbf{x_1}$. Moreover, at this moment the DFE becomes unstable while the EE is locally asymptotically stable. When $\beta = d + \varepsilon\alpha^*$ (i.e., when $k=\alpha^*$), another transcritical bifurcation occurs with $\mathbf{x_2}$ collapsing on $\mathbf{x_0}$ and hence ceasing to exist in the biologically relevant region. Interestingly, the $u$-coordinate of the EE $\mathbf{x_2}$ travels $\mathcal{O}(1)$-distances (from $1-\frac{1}{\alpha}$ to $0$) for a $\mathcal{O}(\varepsilon)$-variation of $\beta$ (from $d+\varepsilon$ to $d+\varepsilon \alpha^*$). See also Figure \ref{fig:end_eq_grad} in Section \ref{sec:numerics} for a numerical exploration of this phenomenon.

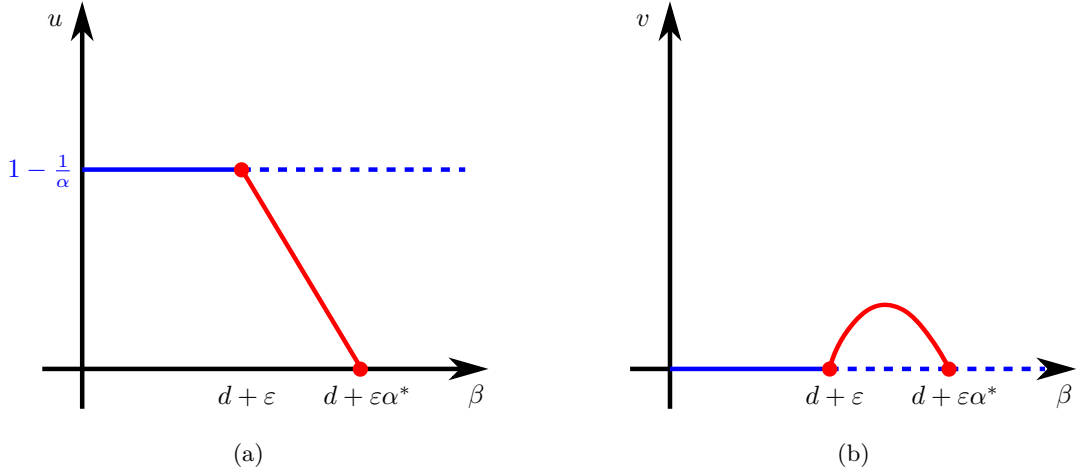
\begin{figure}[h!]
    \centering
\begin{subfigure}{.45\textwidth}
  \centering   
\begin{tikzpicture}
 \node at (0,0) {\includegraphics[width=.8\linewidth]{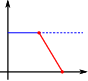}};
\node at (2.8,-2.5) {$\beta$};
\node at (-2.8,2.5) {$u$};
\node at (-3,0.5) {\textcolor{blue}{$1-\frac{1}{\alpha}$}};
\node at (-0.25,-2.5) {$d+\varepsilon$};\node at (1.35,-2.5) {$d+\varepsilon\alpha^*$};
  \end{tikzpicture}
  \caption{}
\end{subfigure}\hspace{0.5cm}
\begin{subfigure}{.45\textwidth}
  \centering
\begin{tikzpicture}
 \node at (0,0) {\includegraphics[width=.8\linewidth]{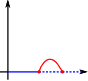}};
\node at (2.8,-2.5) {$\beta$};
\node at (-2.8,2.5) {$v$};
\node at (-0.25,-2.5) {$d+\varepsilon$};\node at (1.35,-2.5) {$d+\varepsilon\alpha^*$};
  \end{tikzpicture}
  \caption{}
\end{subfigure}
    \caption{Bifurcation analysis of system \eqref{eq:rescaled} with $\mathcal{R}_d>0$ (i.e., $\alpha > 1$) and focusing on the parameter $\beta$. Solid lines correspond to stable equilibria, while dashed lines to unstable ones. (a) $u$-coordinate of $\mathbf{x_1}$ and $\mathbf{x_2}$, the red curve related to $\mathbf{x_2}$ is a straight line up to a $\mathcal{O}(\varepsilon)$-factor. Note that the $u$-coordinate of $\mathbf{x_2}$ travels a $\mathcal{O}(1)$ interval (from $1-\frac{1}{\alpha}$ to $0$) in the span of an $\mathcal{O}(\varepsilon)$ interval of $\beta$ (from $d+\varepsilon$ to $d+\varepsilon\alpha^*$; recall \eqref{eq:k}); (b) $v$-coordinate of $\mathbf{x_1}$ and $\mathbf{x_2}$, the red curve related to $\mathbf{x_2}$ is a parabola up to a $\mathcal{O}(\varepsilon^2)$-factor. See also Figure \ref{fig:end_eq_grad} for four numerical simulations of this phenomenon.   \label{fig:bifurcation} }
\end{figure}

\subsection{Nullclines} \label{section:nullclines}

The $v$-nullclines of system \eqref{eq:rescaled} are the $u$-axis and the straight line $v=(\mathcal{R}_0 - 1)u$. The latter distinguishes between the region where $\dot{v}>0$ and where $\dot{v}<0$. In particular, if $\mathcal{R}_0 < 1$ then $\dot{v} \le 0$ for all $(u,v) \in \Delta$ \eqref{eq:Delta}, and $\dot{v}=0$ if and only if $v=0$. The analysis of the $u$-nullclines is more complex, however we can easily see that they necessarily lie $\mathcal{O}(\varepsilon)$-close to the two axes. Indeed, in order to have $\dot{u}=0$ in \eqref{eq:rescaled}, if $u \in \mathcal{O}(1)$ then we must have $v \in \mathcal{O}(\varepsilon)$ and, similarly, if $v \in \mathcal{O}(1)$ then we must have $u \in \mathcal{O}(\varepsilon)$. The following lemma describes precisely the shape of these nullclines (see Figure \ref{fig:nullclines}). 

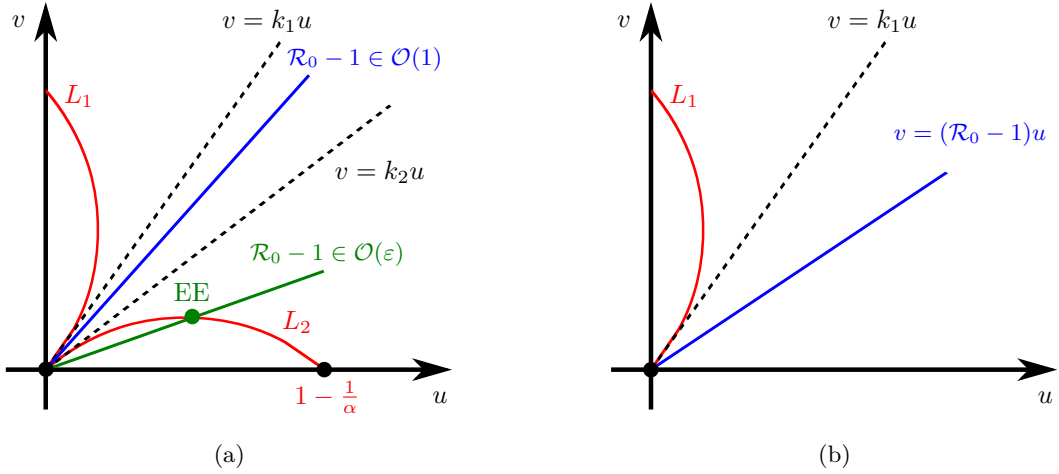
\begin{figure}[h!]
    \centering
\begin{subfigure}{.45\textwidth}
  \centering   
\begin{tikzpicture}
 \node at (0,0) {\includegraphics[width=.8\linewidth]{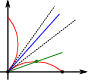}};
\node at (2.8,-2.5) {$u$};
\node at (-2.8,2.5) {$v$};
\node at (-2,1.5) {\textcolor{red}{$L_1$}};
\node at (0.9,-1.5) {\textcolor{red}{$L_2$}};
\node at (1.3,-2.5) {\textcolor{red}{$1-\frac{1}{\alpha}$}};
\node at (-0.5,-1.1) {\textcolor{greenEE}{EE}};
\node at (1.3,-0.6) {\small\textcolor{greenEE}{$\mathcal{R}_0-1 \in \mathcal{O}(\varepsilon)$}};
\node at (1.8,2) {\small\textcolor{blue}{$\mathcal{R}_0-1 \in \mathcal{O}(1)$}};
\node at (0.5,2.5) {$v=k_1u$};\node at (2,0.5) {$v=k_2u$};
  \end{tikzpicture}
  \caption{} \label{fig:nullclines_a} 
\end{subfigure}\hspace{0.5cm}
\begin{subfigure}{.45\textwidth}
  \centering
\begin{tikzpicture}
 \node at (0,0) {\includegraphics[width=.8\linewidth]{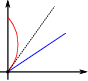}};
\node at (2.8,-2.5) {$u$};
\node at (-2.8,2.5) {$v$};
\node at (-2,1.5) {\textcolor{red}{$L_1$}};
\node at (0.5,2.5) {$v=k_1u$};
\node at (1.8,1) {\small\textcolor{blue}{$v=(\mathcal{R}_0-1)u$}};
  \end{tikzpicture}
  \caption{} \label{fig:nullclines_b} 
\end{subfigure}
    \caption{Nullclines of system \eqref{eq:rescaled} for $\theta \ne 0$. In both figures $\mathcal{R}_0>1$, implying that the $v$-nullclines consist of the $u$-axis and the straight line $v=(\mathcal{R}_0 - 1)u$, which distinguish between when $\dot{v}>0$ and when $\dot{v}<0$ (otherwise the $v$-nullcline would consist of only the $u$-axis and we would have $\dot{v}<0$ for all $v>0$). On the other hand, $\dot{u}>0$ if and only if $(u,v)$ belongs to the small portion of $\Delta$ contained between the axes and the $u$-nullclines. (a) $\mathcal{R}_d > 0$ and the two cases regarding the value attained by $\mathcal{R}_0$ are shown, note that the EE $\mathbf{x_2}$ exists if and only if $\mathcal{R}_0 - 1 \in \mathcal{O}(\varepsilon)$. The $u$-nullcline inside $\Delta$ consists of two branches $L_1$ and $L_2$. (b) $\mathcal{R}_d < 0$ implying that neither $\mathbf{x_1}$ nor $\mathbf{x_2}$ exist. The $u$-nullcline consists of only one branch $L_1$. \label{fig:nullclines} } 
\end{figure}

\begin{lemma} \label{lemma:nullclines}
    Assume that $\theta \ne 0$. If $\mathcal{R}_d > 0$, then the $u$-nullcline inside $\Delta$ \eqref{eq:Delta} consists of two branches $L_1$ and $L_2$ with the following properties: 
    \begin{itemize}
        \item $L_1$ is located in the region $D_1 \coloneqq \{(u,v) \in \Delta \colon \; v > k_1 u\}$ with $k_1>0$, $k_1 \in \mathcal{O}(1/\varepsilon)$; $L_1$ is concave in the $u$-direction; $L_1$ intersects the $v$-axis in the origin and in $(0,1)$; 
        \item $L_2$ is located in the region $D_2 \coloneqq \{(u,v) \in \Delta \colon \; v < k_2 u\}$ with $k_2>0$, $k_2 \in \mathcal{O}(\varepsilon)$; $L_2$ is concave in the $v$-direction; $L_2$ intersects the $u$-axis in the origin and in the DFE $\mathbf{x_1}$. 
    \end{itemize}
    On the other hand, if $\mathcal{R}_d < 0$ then the $u$-nullcline inside $\Delta$ \eqref{eq:Delta} consists of one branch with the same properties of $L_1$ above. 

    Suppose that $\theta=0$. If $\mathcal{R}_d > 0$, then the $u$-nullcline inside $\Delta$ consists of the $v$-axis and one branch with the same properties of $L_2$ above. On the other hand, if $\mathcal{R}_d < 0$ then the $u$-nullcline consists of only the $v$-axis.  
\end{lemma}
\begin{proof}
Let us start by assuming $\theta \ne 0$. We have to determine for which $(u,v) \in \Delta$ we have $\dot{u}=0$ \eqref{eq:rescaled}. Following \cite{li2012global}, this is equivalent to solving 
    \begin{equation} \label{eq:terzo_grado}
        \varepsilon \alpha (u + \theta v) (u+v)^2 = \varepsilon u^2 (\alpha - 1) + uv (-\beta + \varepsilon \alpha + \varepsilon \alpha \theta - \varepsilon) + \varepsilon \alpha \theta v^2. 
    \end{equation}
    We would like to write the right hand side of the previous equation as $\varepsilon \alpha \theta (v - k_1 u)(v - k_2 u)$ for some $k_1, k_2 \in \mathbb{R}$. Namely, we have to determine if the system 
    \begin{align} 
    \left\{
    \begin{aligned}
        & \varepsilon \alpha \theta k_1 k_2 = \varepsilon (\alpha -1) , \\ 
        & -\varepsilon \alpha \theta (k_1 + k_2) = -\beta + \varepsilon \alpha + \varepsilon \alpha \theta - \varepsilon, \\ 
    \end{aligned}
    \right.
    \quad \Longleftrightarrow \quad 
    \left\{
    \begin{aligned}
        & k_1 k_2 = c_1 \in \mathcal{O}(1), \\ 
        & k_1 + k_2 = c_2 \in \mathcal{O}(1/\varepsilon), \\ 
    \end{aligned}
    \right.
    \end{align}
    admits (at least) a solution. The constants $c_1$ and $c_2$ are directly obtained from the coefficients appearing in the system on the left, note that $\textnormal{sign}(c_1)=\textnormal{sign}(\mathcal{R}_d)$ while $c_2 > 0$. Thanks to the difference in order of magnitude between $c_1$ and $c_2$ it is easy to show that a solution to the previous system $(k_1, k_2)$, such that $k_1 \in \mathcal{O}(1/\varepsilon)$ and $k_2 \in \mathcal{O}(\varepsilon)$, always exists. 

    Assume that $\mathcal{R}_d > 0$, i.e., $\alpha > 1$. Since $c_1 > 0$ then $0 < k_2 < k_1$. Hence we have to solve 
    \begin{equation} 
         \left(\frac{1}{\theta}u + v\right) (u+v)^2 = (v - k_1 u)(v - k_2 u), 
    \end{equation} \label{eq:dim_teo_polynomial}
    which can be written in polar coordinates as 
    \begin{equation} \label{eq:dim_teo_rho}
        \rho = \frac{(\sin \varphi - k_2 \cos \varphi)(\sin \varphi - k_1 \cos \varphi)}{\left(\frac{1}{\theta} \cos \varphi + \sin \varphi\right)(\cos \varphi + \sin \varphi)^2}, 
    \end{equation}
    where $u= \rho \cos \varphi$ and $v = \rho \sin \varphi$. From \eqref{eq:dim_teo_rho}, assuming $\varphi \in (0,\frac{\pi}{2})$, it follows that $\rho > 0$ if and only if $\varphi \in (0, \arctan k_2) \cup (\arctan k_1, \frac{\pi}{2})$, hence the nullcline consists of two branches $L_1$ and $L_2$ respectively located in the sets $D_1$ and $D_2$ (defined in the statement of this lemma). A simple calculation shows that $L_1$ intersects the $v$-axis in the origin and in $(0, 1)$, while $L_2$ intersects the $u$-axis in the origin and in the DFE $\mathbf{x_1}$. 

    Suppose now that $\mathcal{R}_d < 0$, i.e., $\alpha < 1$. Since $c_1 < 0$, then $k_2 < 0 < k_1$ with $|k_2| < |k_1|$. Now from \eqref{eq:dim_teo_rho} it follows that $\rho > 0$ if and only if $\varphi \in (\arctan k_1, \frac{\pi}{2})$, hence the nullcline consists of only one branch $L$ with the same properties of the $L_1$ described above. 

    Finally, the properties of concavity of each branch follow from \cite[Lemma 7.1]{li2012global}, which can be adapted to deal with both cases $0 < k_2 < k_1$ and $k_2 < 0 < k_1$, $|k_2| < |k_1|$ considered by us. 

    Consider now the case $\theta=0$. Equation \eqref{eq:terzo_grado} simplifies and it is clear that $u=0$ satisfies the equality for any $v$. Note that, after dividing by $u$, such expression is a polynomial depending on the variables $u$ and $v$ of degree two. The quadratic formula provides an expression of the solutions (note that only the sign $+$ in front of the square root is admissible because we need $v>0$): after expanding the square root with respect to $\varepsilon$, long but simple calculations show that the $u$-nullcline satisfies 
    \begin{equation}
        v = - \frac{\varepsilon}{\beta} u (\alpha u + 1 - \alpha) + \mathcal{O}(\varepsilon^2), 
    \end{equation}
    i.e., it is a parabola (up to a $\mathcal{O}(\varepsilon^2)$-factor). Clearly such parabola lies inside $\Delta$ if and only if $\mathcal{R}_d > 0$. Finally, a simple calculation shows that $(u,v)=\mathbf{x_1}$ satisfies \eqref{eq:terzo_grado} (even without neglecting any term in $\varepsilon$). 
\end{proof}

Concerning a possible intersection between the vertical and the horizontal nullclines, it exists only in the particular case $\mathcal{R}_d > 0$ and $\mathcal{R}_0 = 1 + \mathcal{O}(\varepsilon)$. In this case, the intersection is between the straight line $v=(\mathcal{R}_0 - 1)u$ and the branch $L_2$, and gives life to the EE $\mathbf{x_2}$ (see Figure \ref{fig:nullclines_a}). Note that it is not possible to have an intersection between $v=(\mathcal{R}_0 - 1)u$ and $L_1$. Indeed, this latter branch is $\mathcal{O}(\varepsilon)$-close to the $v$-axis, implying that we would need to have $(\mathcal{R}_0 - 1)\in \mathcal{O}(1/\varepsilon)$. However, it is not possible because $\beta$ and $d$ are of the order $\mathcal{O}(1)$ by assumption of the model. 

\section{Multiple timescale analysis}\label{sec:mult_time}

In this section, we analyze the two-timescale structure of system \eqref{eq:rescaled} through the lenses of Geometric Singular Perturbation Theory (GSPT). We refer to \cite{40,wechselberger2020geometric,2} for an in depth presentation of the theoretical results that we will exploit and to \cite{allee} for a recent example of the application of such results in a two-dimensional setting. In order to highlight the two-timescale
nature of the system, we write system \eqref{eq:rescaled} as 
\begin{equation} \label{eq:explicit_timescales} 
\dot{\mathbf{x}} = \mathbf{F_1}(\mathbf{x}) + \varepsilon \mathbf{F_2}(\mathbf{x}),     
\end{equation}
where 
\begin{equation} \label{eq:terms}
    \mathbf{x}=
    \begin{bmatrix}
        u \\ v
    \end{bmatrix}, \quad
    \mathbf{F_1}(\mathbf{x})=
    \begin{bmatrix}
        -\beta \frac{u v}{u+v}\\ 
        \beta \frac{u v}{u+v} - dv
    \end{bmatrix}, \quad
    \mathbf{F_2}(\mathbf{x})=
    \begin{bmatrix}
        \alpha (u + \theta v)(1-u-v) - u \\ 
        - v
    \end{bmatrix}. 
\end{equation}

\subsection{Fast formulation}

Setting $\varepsilon=0$ in \eqref{eq:explicit_timescales}, we obtain the corresponding fast subsystem as
\begin{align} \label{eq:fast_subsystem}
\begin{aligned}
    \dot{u} &= -\beta \frac{u v}{u + v}, \\ 
    \dot{v} &= \beta \frac{u v}{u + v} - dv. 
\end{aligned}
\end{align}
The critical manifold $\mathcal{C}_0$ of \eqref{eq:explicit_timescales} is defined as the set of the equilibria of \eqref{eq:fast_subsystem}, namely 
\begin{equation} \label{eq:critical_manifold}
    \mathcal{C}_0 \coloneqq \{(u,v) \in \mathbb{R}^2 \colon \; v=0\}. 
\end{equation}
Denote with $(u_0,v_0)$, $u_0, v_0 >0$ the initial conditions and define $u_\infty$ and $v_\infty$ as 
\begin{equation}
    u_\infty=\lim_{t \to +\infty} u(t) \quad \textnormal{and} \quad v_\infty=\lim_{t \to +\infty} v(t), 
\end{equation}
when these limit exist under the flow of system \eqref{eq:fast_subsystem}. 

\begin{proposition} \label{prop:existence}
The trajectories of system \eqref{eq:fast_subsystem} converge to $\mathcal{C}_0$ as $t \to +\infty$. 
\end{proposition}
\begin{proof}
Note that the solutions of the fast system \eqref{eq:fast_subsystem} naturally evolve inside $\Delta$ \eqref{eq:Delta}. Since $\dot{u} \le 0$, there exists $u_\infty \in [0, u_0]$. Moreover, $\dot{u}+\dot{v} \le 0$, therefore there exist also $v_\infty \ge 0$. Integrating $\dot{u}+\dot{v}$ we obtain 
    \begin{equation*}
    \begin{split}
        - \infty  < u_\infty + v_\infty - u_0 - v_0 = \int_0^{+\infty} \left(\dot{u}(t)+\dot{v}(t)\right) \; \dd t  = -d \int_0^{+\infty} v(t) \;  \dd t  < 0, 
    \end{split}
    \end{equation*}
    therefore $v_\infty=0$. 
\end{proof} 

The following proposition provides an explicit expression for $u_\infty$. 

\begin{proposition} \label{prop:Gamma}
The quantity 
\begin{equation} \label{eq:constant_motion}
    \Gamma(u,v)\coloneqq (u+v) u^{-\frac{d}{\beta}}
\end{equation}
is a constant of motion for system \eqref{eq:fast_subsystem} as long as $u>0$. Moreover, if $\beta < d$ then 
\begin{equation} \label{eq:u_infty}
    u_\infty = u_0 \left(1+ \frac{v_0}{u_0}\right)^{\frac{\beta}{\beta - d}}
\end{equation}
and $v \to 0$ exponentially fast, otherwise if $\beta\geq  d$ then $u_\infty=0$ and $u \to 0$ at least exponentially fast if $\beta>d$.
\end{proposition}
\begin{proof}
By direct derivation with respect to the time variable $t$, one can see that $\dot{\Gamma}(u,v) \equiv 0$. However, since it is not trivial to obtain such function, we show how to derive it. Since, as long as $u>0$,  
\begin{equation*}
    \frac{\dd v}{\dd u} = \frac{\dot{v}}{\dot{u}} = \frac{d}{\beta} \frac{v}{u} + \frac{d}{\beta} - 1, 
\end{equation*}
it is convenient to introduce the variable $w=v/u$ (note that we can assume $u \not= 0$ since the set $\{u=0\}$ is invariant for the orbits of \eqref{eq:fast_subsystem}, on which they approach the origin). Hence we also have 
\begin{equation*}
    \frac{\dd v}{\dd u} = \frac{\dd (u w)}{\dd u} = w + u \frac{\dd w}{\dd u}, 
\end{equation*}
implying that 
\begin{equation*}
    \frac{d}{\beta} w + \frac{d}{\beta} - 1 = w + u \frac{\dd w}{\dd u}, 
\end{equation*}
which in turn is equivalent to 
\begin{equation*}
    \frac{\dd w}{1 + w} = \left(\frac{d}{\beta}-1 \right) \frac{\dd u}{u}. 
\end{equation*}
Integrating both sides of the previous equation and going back to the original variables $u$ and $v$, one recovers the quantity $\Gamma(u,v)$ \eqref{eq:constant_motion}. 

Concerning the values attained by $u_\infty$, let us start assuming $\beta \ge d$. A simple calculation shows that 
\begin{equation*}
    \frac{\dd}{\dd t} \frac{v(t)}{u(t)} = (\beta - d) \frac{v(t)}{u(t)}, 
\end{equation*}
implying that 
\begin{equation} \label{eq:v}
    v(t) = \frac{v_0}{u_0} \exp((\beta - d) t) u(t), 
\end{equation}
hence $u(t) \to 0$ (at least exponentially fast if $\beta>d$) because $v_\infty=0$. On the other hand, if $\beta < d$ we cannot derive the limit of $u(t)$ from \eqref{eq:v}, but it shows that $v(t) \to 0$ exponentially fast.  However, from \eqref{eq:constant_motion} we can directly derive \eqref{eq:u_infty} since, again, $v_\infty=0$. Note that \eqref{eq:u_infty} does not hold in the case $\beta \geq d$ because $u(t) \to 0$ and $\Gamma$ is only well-defined as long as $u>0$.   
\end{proof}

Note that \eqref{eq:u_infty} implies that $u_\infty \le u_0$, according to the proof of Proposition \ref{prop:existence}. A simple calculation shows that the eigenvalue $\lambda$ associated to the fast variable $v$ of \eqref{eq:rescaled} evaluated on the critical manifold $\mathcal{C}_0$ \eqref{eq:critical_manifold} is $\lambda=\beta - d$, implying, if $\beta \not= d$, that $\mathcal{C}_0$ is attractive or repelling everywhere \cite[Chapter 3]{40}. This is in line with Proposition \ref{prop:Gamma}, which shows that the orbits of \eqref{eq:fast_subsystem} cannot converge toward $\mathcal{C}_0$ if $\beta > d$ since, in this case, the manifold is repelling. In this situation, $\mathbf{x_0}$ represents an exception because systems \eqref{eq:rescaled} and \eqref{eq:fast_subsystem} are not well-defined there, implying that particular dynamics may occur around it. Indeed, the regularity properties of the functions appearing in such systems, needed to apply the standard analytical results, are not satisfied at $\mathbf{x_0}$. Moreover, observe that $\mathcal{C}_0$ is attractive if $\mathcal{R}_0 < 1$, while it is repelling if $\mathcal{R}_0 > 1$, indeed $\mathcal{R}_0 \to \beta/d$ as $\varepsilon \to 0$. This was to be expected since, by definition, if $\mathcal{R}_0 < 1$ then the parasites will die out. Finally, from Proposition \ref{prop:Gamma} it follows that if $\beta=d$ (i.e., $\lambda=0$ and $\mathcal{R}_0 \to 1$ as $\varepsilon \to 0$) then the orbits behave like in the case $\beta > d$, implying that $\mathcal{C}_0$ is repelling also in this case. 

\subsection{Slow formulation} \label{section:slow}

Consider \eqref{eq:rescaled} and assume that a solution reached an $\mathcal{O}(\varepsilon^2)$-neighborhood of the critical manifold $\mathcal{C}_0$, namely $v \in \mathcal{O}(\varepsilon^2)$, for $u > 0$. In this situation, the influence of $\mathbf{F_2}$ \eqref{eq:terms} becomes very relevant. We rescale $v$ as  $v=\varepsilon x$ and apply a rescaling to the time variable, bringing the system to the slow timescale $\tau=\varepsilon t$: 
\begin{align} \label{eq:slow_subsystem}
\begin{aligned}
    u' &= \alpha (u+ \varepsilon \theta x)(1-u-\varepsilon x) - u - \beta \frac{u x}{u + \varepsilon x}, \\ 
    \varepsilon x' &= \beta \frac{u x}{u + \varepsilon x} - d x - \varepsilon x, 
\end{aligned}
\end{align}
where the $'$ indicates the derivative with respect to the slow time $\tau$. Note that $v \in \mathcal{O}(\varepsilon^2)$ implies $x \in \mathcal{O}(\varepsilon)$. 

If we look at system \eqref{eq:slow_subsystem} on the critical manifold $\mathcal{C}_0$, now determined by $x=0$, we obtain 
\begin{equation} \label{eq:slow_flow}
    u' = \alpha u(1-u) - u.  
\end{equation}
Hence, two scenarios must be distinguished. If $\alpha < 1$ then, on $\mathcal{C}_0$, $u$ decreases toward $0$. This was to be expected since $\alpha < 1$ implies $\mathcal{R}_d < 0$, meaning that the population would become extinct even in the absence of the parasite. On the other hand, if $\alpha > 1$ then, on $\mathcal{C}_0$, $u$ converges toward $1-1/\alpha$. This means that it asymptotically approaches the DFE $\mathbf{x_1}$. Note that this behavior on the critical manifold is independent of $\mathcal{R}_0$, indeed $\mathbf{x_1}$ always has a stable manifold, corresponding to the $u$-axis, related to its eigenvalue $\lambda_1<0$ (recall Section \ref{section:equilibria}). 

Since system \eqref{eq:slow_subsystem} is in standard form and as long as the critical manifold is normally hyperbolic (i.e., the eigenvalue $\lambda = \beta - d$ associated to the fast variable $x$ is nonzero), we can apply Fenichel's Theorem \cite[Theorem 3.1.4]{40} to understand the behavior of its orbits close to $\mathcal{C}_0$ and away from $\mathbf{x_0}$ (which represents a singularity). Fenichel's theorem implies that, as long as $x \in \mathcal{O}(\varepsilon)$, the system evolves on the slow timescale and, in the limit $\varepsilon \to 0$, $u$ behaves like the solution to \eqref{eq:slow_flow}. Indeed, note that if $\varepsilon \to 0$ then $x \to 0$ and the evolution of $u$ \eqref{eq:slow_subsystem} converges to the evolution described by \eqref{eq:slow_flow}. The fact that this happens if $v$ is $\mathcal{O}(\varepsilon^2)$-close to $\mathcal{C}_0$ ensures that the slow flow is not affected by the EE $\mathbf{x_2}$ (when it exists) which is $\mathcal{O}(\varepsilon)$-away from $\mathcal{C}_0$. Clearly, if the critical manifold is repelling (i.e., $\mathcal{R}_0 > 1$) then the orbits will escape from the slow flow. On the other hand, if it is attractive (i.e., $\mathcal{R}_0 < 1$) they will stay close to $\mathcal{C}_0$ and converge to $\mathbf{x_0}$ if $\mathcal{R}_d < 0$ or to $\mathbf{x_1}$ if $\mathcal{R}_d > 0$. Note that these behaviors are in perfectly agreement with the analysis of the nullclines carried out in Section \ref{section:nullclines}. 

\subsection{Unified formulation} \label{section:unified}

In this section, our objective is to combine the analysis of the previous two sections, together with other results reported in Appendix \ref{section:blowup}, summarized in the following theorem, to understand the complete behavior of our model. 

\begin{theorem} \label{teo:summary}
Consider \eqref{eq:rescaled}. Then, in the biologically relevant region $\Delta$ \eqref{eq:Delta}, the singular equilibrium $\mathbf{x_0}$ has the following properties:
\begin{itemize}
    \item If $\alpha<1$ and $0<\beta-d\in\mathcal O(1)$, then $\mathbf{x_0}$ is a topological sink.
    \item If $\alpha>1$ and $\beta>d+\varepsilon \alpha^*$, then $\mathbf{x_0}$ has a sector foliated by homoclinic orbits to itself.
    \item  If $\alpha>1$ and $d+\varepsilon<\beta<d+\varepsilon\alpha^*$, then $\mathbf{x_0}$ is a topological saddle with the $v$-axis attracting and the $u$-axis repelling.
\end{itemize}
\end{theorem}

To avoid breaking the flow of the exposition, we have placed the proof of this theorem in Appendix \ref{section:blowup}. Now, we start by proving the following lemma, which implies that system \eqref{eq:rescaled} does not admit any limit cycles.

\begin{lemma} \label{lemma:Dulac}
    System \eqref{eq:rescaled} has no closed trajectories lying in the interior of $\Delta$ \eqref{eq:Delta}. 
\end{lemma}
\begin{proof}
    Define $D(u,v)=1/(uv)$ \cite{berezovskaya2004simple,li2012global}, then 
\begin{equation}
    \frac{\partial (D(u,v) \dot{u})}{\partial u} + \frac{\partial (D(u,v) \dot{v})}{\partial v} = \varepsilon \alpha \left((1-\theta)\frac{v}{(u+v)^2} - 1\right) < 0
\end{equation}
for all $(u,v)$ lying in the interior of $\Delta$. Hence $D(u,v)$ is a Dulac function and the thesis follows.  
\end{proof} 

Standard perturbation theory \cite[Corollary 3.1.7]{2} implies that an orbit of the perturbed system \eqref{eq:rescaled}, away from the critical manifold $\mathcal{C}_0$, follows $\mathcal{O}(\varepsilon)$-closely the orbit of the fast system \eqref{eq:fast_subsystem}, related to the same initial conditions, for $\mathcal{O}(1)$ times $t$. However, several situations must be distinguished, according to the analysis of the bifurcations and of the nullclines of our system (see Sections \ref{section:equilibria} and \ref{section:nullclines}), to be able to fully analyze model \eqref{eq:rescaled}. Figure \ref{fig:sketch} represents all possible scenarios; in the remainder of this section, we discuss them in detail. We will not distinguish the cases $\theta \in (0,1]$ and $\theta = 0$ because the only difference lies in the $L_1$ branch of the $u$-nullcline, which will not play a crucial role in determining the behavior of the orbits. 

\begin{figure}[h!]
\centering
\begin{subfigure}{.3\textwidth}
  \centering
  \begin{tikzpicture}
 \node at (0,0) {\includegraphics[width=.85\linewidth]{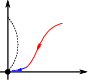}};
\node at (2.2,-1.6) {$u$};
\node at (-2,1.8) {$v$};
\node at (-1.2,0.8) {$L_1$};
\node at (-2,-1.8) {$\mathbf{x_0}$};
\node at (0,2.5) {Case 1: $\mathcal{R}_0 < 1$, $\mathcal{R}_d < 0$};
  \end{tikzpicture}
  \caption{}
  \label{fig:sketch1}
\end{subfigure}\hspace{.5cm}
\begin{subfigure}{.3\textwidth}
  \centering
  \begin{tikzpicture}
 \node at (0,0) {\includegraphics[width=.85\linewidth]{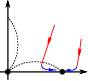}};
\node at (2.2,-1.6) {$u$};
\node at (-2,1.8) {$v$};
\node at (-1.2,0.8) {$L_1$};
\node at (0.7,-1.1) {$L_2$};
\node at (0.9,-1.8) {$\mathbf{x_1}$};
\node at (-2,-1.8) {$\mathbf{x_0}$};
\node at (0,2.5) {Case 2: $\mathcal{R}_0 < 1$, $\mathcal{R}_d > 0$};
  \end{tikzpicture}
  \caption{}
  \label{fig:sketch4}
\end{subfigure}\hspace{.5cm}
\begin{subfigure}{.3\textwidth}
  \centering
 \begin{tikzpicture}
 \node at (0,0) {\includegraphics[width=.85\linewidth]{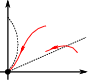}};
\node at (2.2,-1.6) {$u$};
\node at (-2,1.8) {$v$};
\node at (-1.2,0.8) {$L_1$};
\node at (-2,-1.8) {$\mathbf{x_0}$};
\node at (2,0.3) {$v=mu$};
\node at (0,2.5) {Case 3: $\mathcal{R}_0 > 1$, $\mathcal{R}_d < 0$};
  \end{tikzpicture}
  \caption{}
  \label{fig:sketch2}
\end{subfigure} \\ 
\vspace{.5cm} 
\begin{subfigure}{.3\textwidth}
  \centering
 \begin{tikzpicture}
 \node at (0,0) {\includegraphics[width=.85\linewidth]{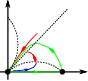}};
\node at (2.2,-1.6) {$u$};
\node at (-2,1.8) {$v$};
\node at (-1.2,0.8) {$L_1$};
\node at (0.5,-0.9) {$L_2$};
\node at (1.5,1) {$v=mu$};
\node at (0.9,-1.8) {$\mathbf{x_1}$};
\node at (-2,-1.8) {$\mathbf{x_0}$};
\node at (0,2.5) {Case 4: $0<\mathcal{R}_0 - 1 \in \mathcal{O}(1)$, $\mathcal{R}_d > 0$};
  \end{tikzpicture}
  \caption{}
  \label{fig:sketch5}
\end{subfigure} \hspace{1.5cm} 
\begin{subfigure}{.3\textwidth}
  \centering
 \begin{tikzpicture}
 \node at (0,0) {\includegraphics[width=.85\linewidth]{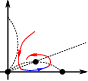}};
\node at (2.2,-1.6) {$u$};
\node at (-2,1.8) {$v$};
\node at (-1.2,0.8) {$L_1$};
\node at (0.7,-1.1) {$L_2$};
\node at (0.9,-1.8) {$\mathbf{x_1}$};
\node at (-0.5,-1.25) {$\mathbf{x_2}$};
\node at (-2,-1.8) {$\mathbf{x_0}$};
\node at (2,0.1) {$v=mu$};
\node at (0,2.5) {Case 5: $0<\mathcal{R}_0 - 1 \in \mathcal{O}(\varepsilon)$, $\mathcal{R}_d > 0$};
  \end{tikzpicture}
  \caption{}
  \label{fig:sketch3}
\end{subfigure}
\caption{Five possible behaviors of the dynamics of system \eqref{eq:rescaled} assuming $\theta \ne 0$ (if $\theta=0$ then the only difference lies in the fact that $L_1$ would collapse on the $v$-axis). For ease of notation, we set $m=\mathcal{R}_0-1$ when $\mathcal{R}_0>1$. (a) $\mathcal{R}_0 < 1$, $\mathcal{R}_d < 0$: the origin $\mathbf{x_0}$ is globally asymptotically stable, so the population goes extinct, through a fast piece of orbit followed by a slow piece close to the critical manifold; (b) $\mathcal{R}_0 < 1$, $\mathcal{R}_d > 0$: the parasites go extinct, as the DFE $\mathbf{x_1}$ is globally asymptotically stable. The curve $v=\gamma(u)$ defined in \eqref{eq:curve_gamma} allows us to predict whether orbits will approach $\mathbf{x_1}$ in the slow flow from the left or from the right (see also Figure \ref{fig:test_gamma}); (c) $\mathcal{R}_0 > 1$, $\mathcal{R}_d < 0$: the origin $\mathbf{x_0}$ is globally asymptotically stable, so the population goes extinct, possibly after a first peak in the parasite population; (d) To be precise, this scenario corresponds to $\mathcal{R}_d > 0$, $\beta>d+\varepsilon \alpha^*$ \eqref{eq:k}; as long as this requirement is satisfied, we can allow also $\mathcal{R}_0-1 \in \mathcal{O}(\varepsilon)$. The space within the heteroclinic orbit between $\mathbf{x_0}$ and $\mathbf{x_1}$ (green) is foliated by homoclinic orbits to $\mathbf{x_0}$ (see also Figure \ref{fig:homoclinic} and Appendix \ref{sec:blow-up_d}); (e) To be precise, this scenario corresponds to $\mathcal{R}_d > 0$, $d+\varepsilon < \beta<d+\varepsilon \alpha^*$, which implies that $\mathcal{R}_0-1 \in \mathcal{O}(\varepsilon)$. In this case, the EE $\mathbf{x_2}$ is globally asymptotically stable, although orbits my pass close to the origin $\mathbf{x_0}$ and spend a long time traveling close to the repelling critical manifold before reaching a neighborhood of $\mathbf{x_2}$ (see also Figure \ref{fig:full_dyn} and Appendix \ref{sec:blow-up_e}, and Figure \ref{fig:end_eq_grad} to see that $\mathbf{x_2}$ only exists for a $\mathcal{O}(\varepsilon)$-large interval of $\beta$ values). \label{fig:sketch}}\end{figure}
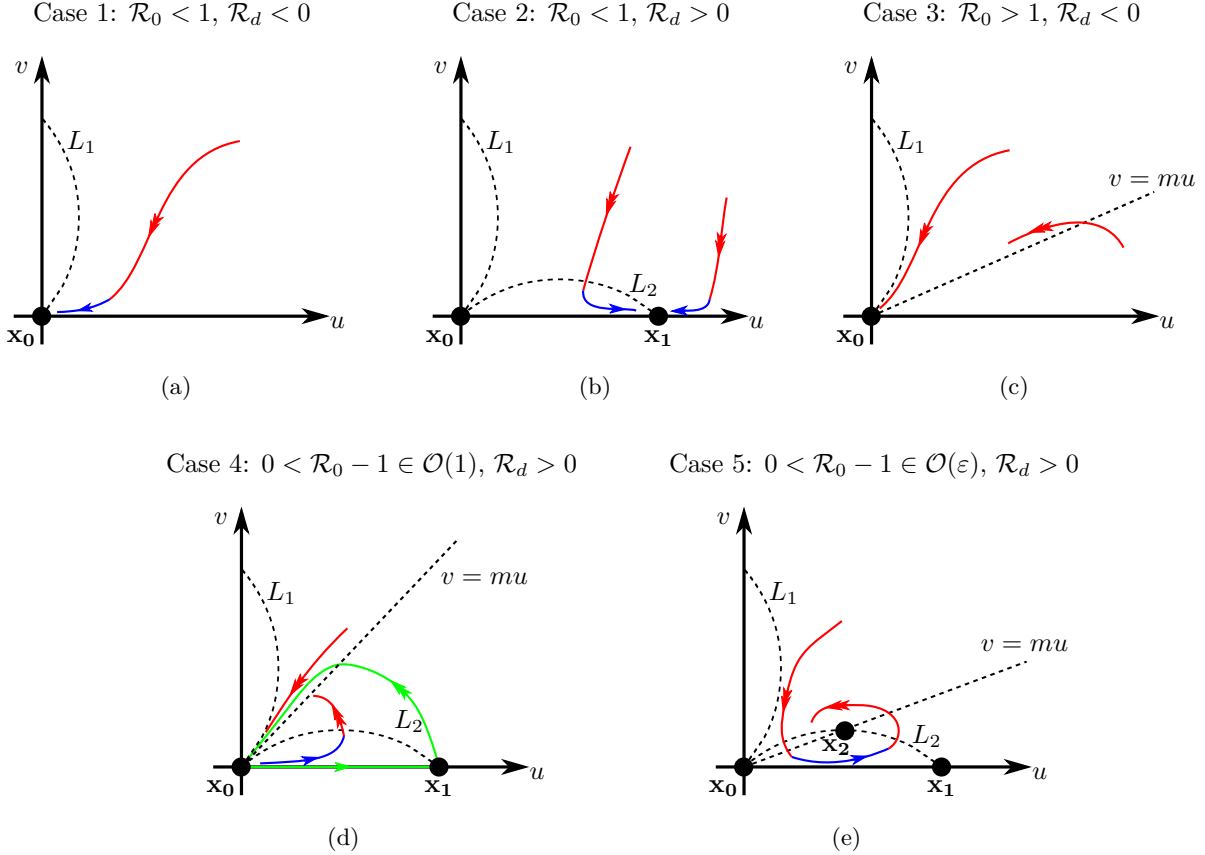

\subsubsection{$\mathcal{R}_0 < 1$ and $\mathcal{R}_d < 0$}

In this situation, $\dot{v}<0$ for all $(u,v)$ with $v>0$, and the critical manifold $\mathcal{C}_0$ is attractive. Recall \eqref{eq:u_infty}; an orbit starting from $(u_0,v_0)$ away from $\mathcal{C}_0$ enters the slow flow for some 
\begin{equation}
    u=u_{\infty,\varepsilon} \approx u_\infty=\left(1+\frac{v_0}{u_0}\right)^{\frac{\beta}{\beta-d}}>0.
\end{equation}
As proven in Section \ref{section:slow}, then such orbit converges to $\mathbf{x_0}$. This is not surprising since $\mathcal{R}_d < 0$ means that the population will become extinct even without the inclusion of the parasite. This situation is illustrated in Figure \ref{fig:sketch1}. 

\subsubsection{$\mathcal{R}_0 < 1$ and $\mathcal{R}_d > 0$}

Similarly to the previous case, here $\dot{v}<0$ for all $(u,v)$ with $v>0$, and the critical manifold $\mathcal{C}_0$ is attractive. However, now there exists the DFE $\mathbf{x_1}$ and, if the parasites were not present, the population would grow because $\mathcal{R}_d > 0$. As in the previous scenario, an orbit starting from $(u_0,v_0)$ away from $\mathcal{C}_0$ enters the slow flow for some 
\begin{equation}
    u=u_{\infty,\varepsilon} \approx u_\infty=\left(1+\frac{v_0}{u_0}\right)^{\frac{\beta}{\beta-d}}>0.
\end{equation}
Moreover, as proven in Section \ref{section:slow}, such orbit then converges to the DFE $\mathbf{x_1}$, with $u$ increasing during the slow flow if $u_{\infty} < 1-1/\alpha$, or decreasing if $u_{\infty} > 1-1/\alpha$. In the limit $\varepsilon \to 0$, these two possibilities are distinguished by the curve 
\begin{equation} \label{eq:curve_gamma}
    v=\gamma(u) \coloneqq u^{\frac{d}{\beta}} \left(1 - \frac{1}{\alpha}\right)^{1-\frac{d}{\beta}} - u, 
\end{equation}
which is obtained by imposing $\Gamma(u,v) = \Gamma(1-1/\alpha, 0)$; recall \eqref{eq:constant_motion}. This situation is illustrated in Figure \ref{fig:sketch4}, and we showcase the predictive accuracy of \eqref{eq:curve_gamma} in Figure \ref{fig:test_gamma}. 

\begin{remark}
    When $\mathcal{R}_0 < 1$ it is possible to determine the attractors of the orbit even without GSPT techniques. Indeed, consider the candidate Lyapunov function $L(u,v)=v$. Then, since 
    \begin{equation}
        \frac{\dd L(u,v)}{\dd t} = v \left( \beta \frac{u}{u+v} -d - \varepsilon \right) \le v (\beta -d-\varepsilon) \le 0, 
    \end{equation}
    standard Lyapunov theory implies that the orbits converge toward the $u$-axis. Then, LaSalle’s Invariance Principle implies that if $\mathcal{R}_d < 0$ the orbits would approach $\mathbf{x_0}$, while if $\mathcal{R}_d > 0$ the orbits would be attracted to the DFE $\mathbf{x_1}$. 
\end{remark}

\subsubsection{$\mathcal{R}_0 > 1$ and $\mathcal{R}_d < 0$}

Proposition \ref{prop:Gamma} implies that, during the fast flow, the orbits approach $\mathbf{x_0}$, with the behavior represented in Figure \ref{fig:sketch2}. The structure of the nullclines implies that later they cannot escape from the attraction of $\mathbf{x_0}$ (recall Theorem \ref{teo:summary}), and hence converge toward it. This is also in line with the fact that since $\mathcal{R}_d < 0$ then the population will necessarily become extinct. 

\begin{remark}
    When $\mathcal{R}_d < 0$ it is clear that the orbits will converge toward $\mathbf{x_0}$, regardless of the value of $\mathcal{R}_0$. Not only because, by definition, $\mathcal{R}_d < 0$ implies that the population will become extinct, but also because $\mathbf{x_0}$ is the only possible attractor of the system (recall Lemma \ref{lemma:Dulac}). However, we exploited GSPT techniques to understand the exact qualitative behavior of the orbits. 
\end{remark}

\subsubsection{$ 0 <\mathcal{R}_0 - 1 \in \mathcal{O}(1)$ and $\mathcal{R}_d > 0$} \label{sec:R01}

More precisely, in this section we assume that $\alpha>1$ (i.e., $\mathcal{R}_d > 0$) and $\beta>d + \varepsilon \alpha^*$ \eqref{eq:k}, implying that the EE $\mathbf{x_2}$ does not exist. Figure \ref{fig:sketch5} illustrates this scenario. 

Similarly to before, in this case the orbits will be necessarily attracted to $\mathbf{x_0}$ since it is the only possible attractor of the system (recall that the DFE $\mathbf{x_1}$ exists but it is unstable). However, our objective is understand the precise behavior of the trajectories and prove the existence of ``slow-fast'' homoclinic orbits. Indeed, in Appendix \ref{sec:blow-up_d} (see also Theorem \ref{teo:summary}), we analytically derive the behavior of the orbits close to $\mathbf{x_0}$, showing that indeed the system exhibits homoclinic orbits to such singular equilibrium (see Figure \ref{fig:blow-up1}). 

Let us focus now on the cycles. It is easy to see that there exists an heteroclinic cycle between $\mathbf{x_0}$ and $\mathbf{x_1}$: the first heteroclinic orbit lies on the $u$-axis and evolves under the slow flow; the second heteroclinic orbit evolves under the fast flow and, in the limit $\varepsilon \to 0$, is described by the curve $v=\gamma(u)$ \eqref{eq:curve_gamma}. Note from Figure \ref{fig:sketch5} that the shape of this curve is completely different from the one in the case $\mathcal{R}_0 < 1$ and $\mathcal{R}_d > 0$. Moreover, note that $\gamma$ approaches $\mathbf{x_0}$ with a vertical tangent, since it ``ignores''  the $L_1$ branch of the $u$-nullcline, as the former is a curve that exists in the fast flow (recall that $L_1$ collapses on the $v$-axis as $\varepsilon \to 0$). However, as shown in Figure \ref{fig:ppK1} (see also Figure \ref{fig:homoclinic}), the heteroclinic orbit approaches the origin with a vertical tangent only if $2d-\beta<0$. This heteroclinic cycle has two important characteristics: it contains an infinite number of homoclinic orbits and no other orbits; it separates those homoclinic orbits with the rest of the trajectories approaching $\mathbf{x_0}$ and starting from the region of the set $\Delta$ outside the heteroclinic cycle. 

We provide now more details regarding the behavior of the homoclinic orbits. These orbits begin their journey in the slow flow but they tend to exit from such flow, and hence enter the fast one, because the critical manifold $\mathcal{C}_0$ is repelling. Assume that this happens for some $u=u_E \in (0, 1-1/\alpha)$, then the remaining fast flow of the orbit can be approximated by the curve 
\begin{equation} \label{eq:gamma_E}
    v = \gamma_{u_E} (u) \coloneqq u^{\frac{d}{\beta}} u_E^{1-\frac{d}{\beta}} - u, 
\end{equation}
which was obtained by imposing $\Gamma(u,v) = \Gamma(u_E, 0)$ \eqref{eq:constant_motion}. Note that we neglected the influence of $v$ since $v \in \mathcal{O}(\varepsilon)$, and that any orbit starting from a point $(u_0,v_0)$ between the $u$-axis and the curve $v=\gamma(u)$ \eqref{eq:curve_gamma} tends to $\mathbf{x_0}$ as $t \to - \infty$. Therefore, the homoclinic orbits can exhibit two different behaviors: they can immediately exit from the slow flow and enter the fast one, rapidly converging toward $\mathbf{x_0}$ as $t \to + \infty$; or they can stay in the slow flow for $\mathcal{O}(1/\varepsilon)$-times $t$ (i.e., $\mathcal{O}(1)$-times $\tau$) and exit from such flow for a large exit point $u_E$. The limit case related to the second scenario is clearly represented by the heteroclinic cycle. Indeed, in Section \ref{sec:numerics} (particularly, Figure \ref{fig:homoclinic}) we showcase the behavior of these orbits, which foliate a region of the biologically relevant region $\Delta$ \eqref{eq:Delta} .

\begin{remark} \label{rem:canard}
    These interesting dynamics of model \eqref{eq:rescaled} follow from the fact that $\mathbf{x_0}$ is a singular equilibrium, allowing the orbits to reach a $\mathcal{O}(\varepsilon^2)$-neighborhood of $\mathcal{C}_0$ even if it is repelling. Then, such orbits can remain close to the critical manifold for $\mathcal{O}(1)$-times $t$, giving life to a sort of canard-like \cite{40} or funneling \cite{yanchuk2026singular} behavior. As we will see, something similar also happens when the EE $\mathbf{x_2}$ exists. 
\end{remark}

\subsubsection{$0 <\mathcal{R}_0 - 1 \in \mathcal{O}(\varepsilon)$ and $\mathcal{R}_d > 0$} \label{sec:R0esp}

More precisely, in this section we assume that $\alpha>1$ (i.e., $\mathcal{R}_d > 0$) and $\beta=d + \varepsilon k$ with $1<k<\alpha^*$ \eqref{eq:k}, implying that the EE $\mathbf{x_2}$ exists. Figure \ref{fig:sketch3} illustrates this scenario.

In the limit $\varepsilon \to 0$, $\beta \to d$ implying that $\mathcal{R}_0 \to 1$. Therefore, according to Proposition \ref{prop:Gamma}, in this case the critical manifold $\mathcal{C}_0$ is repelling. In particular, this implies that also the DFE $\mathbf{x_1}$ is unstable. 

In Appendix \ref{sec:blow-up_e} (see also Theorem \ref{teo:summary}), we show that the origin $\mathbf{x_0}$ is a saddle that repels the orbits approaching it along the $v$-axis. Since also the DFE $\mathbf{x_1}$ is unstable and there are no limit cycles (recall Lemma \ref{lemma:Dulac}), the EE $\mathbf{x_2}$ is necessarily globally asymptotically stable. Concerning the complete behavior of an orbit, it is approximated by $\Gamma$ \eqref{eq:constant_motion} in the fast slow, during which it approaches the origin $\mathbf{x_0}$. As described in Figure \ref{fig:BUe}, such a point repels it, and its $u$-coordinate starts to increase in the slow flow. Interestingly, the orbit shows a \emph{canard-like} behavior, indeed, in order to converge to $\mathbf{x_2}$, it necessarily travels $\mathcal{O}(1)$-distances in the slow flow even if the critical manifold is repelling (recall Remark \ref{rem:canard}). After this long excursion in the slow flow, the fast slow begins again, and this process is repeated (recall \eqref{eq:gamma_E}). Finally, this process stops, and the orbit is attracted to the EE $\mathbf{x_2}$ (which does not interact with the slow flow since it is $\mathcal{O}(\varepsilon)$-away from $\mathcal{C}_0$). 
In Section \ref{sec:numerics} (particularly, Figure \ref{fig:full_dyn}), we showcase the behavior of these orbits when they pass close to $\mathbf{x_0}$. 

\section{Numerical simulations}\label{sec:numerics} 

In order to make the numerical simulation of our model easier, following \cite[Appendix A.5]{della2024geometric} we introduce the variable $w=\log v$ in order to re-write system \eqref{eq:rescaled} as 
\begin{align} \label{eq:numerical_w}
\begin{aligned}
    \dot{u} &= \varepsilon \alpha (u + \theta \exp(w)) (1 - u -\exp(w))  - \varepsilon u - \beta \frac{u \exp(w)}{u + \exp(w)}, \\ 
    \dot{w} &= \beta \frac{u}{u+ \exp(w)} -d - \varepsilon. \\ 
\end{aligned}
\end{align}
System \eqref{eq:numerical_w} is remarkably less stiff than system \eqref{eq:rescaled}, and this logarithmic change of coordinates can be beneficial for the numerical integration of systems with variables becoming really small (while remaining positive). 

We remark that our choice for the parameters is motivated by which values would result in the clearest visualizations of the dynamics we described in the previous sections, rather than their relevance to specific biological systems.

In Figure \ref{fig:test_gamma}, we illustrate how well the curve $v=\gamma(u)$ \eqref{eq:curve_gamma} allows us to predict whether an orbit will enter the slow flow with $u_\infty \lessgtr 1-1/\alpha$, in the case depicted in Figure \ref{fig:sketch4}. Note the increase in prediction power as $\varepsilon$ decreases, according to the fact the such curve was derived in the limit $\varepsilon \to 0$.  

\begin{figure}[h!]
    \centering
\begin{subfigure}{.45\textwidth}
  \centering   
\begin{tikzpicture}
 \node at (0,0) {\includegraphics[width=.8\linewidth]{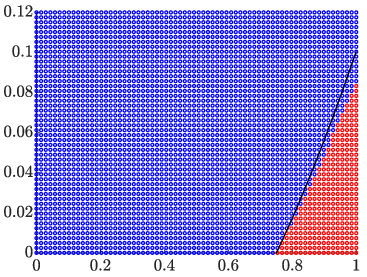}};
\node at (3.1,-1.8) {$u$};
\node at (-2.5,2.3) {$v$};
  \end{tikzpicture}
  \caption{} \label{fig:test_gamma_a} 
\end{subfigure}\hspace{0.5cm}
\begin{subfigure}{.45\textwidth}
  \centering
\begin{tikzpicture}
 \node at (0,0) {\includegraphics[width=.8\linewidth]{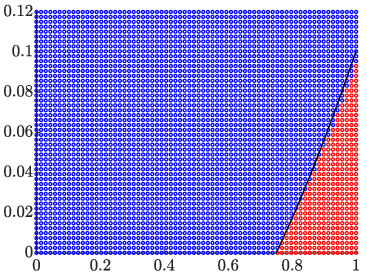}};
\node at (3.1,-1.8) {$u$};
\node at (-2.5,2.3) {$v$};
  \end{tikzpicture}
  \caption{} \label{fig:test_gamma_b} 
\end{subfigure}  \caption{Test on the accuracy of \eqref{eq:curve_gamma} in predicting whether an orbit will enter the slow flow with $u_\infty \lessgtr 1-1/\alpha$. Values of the parameters: $\alpha=4$, $\theta=0.5$, $\beta=0.075$, $d=0.1$ so that $\mathcal{R}_d>0$ and $\mathcal{R}_0<1$, and (a) $\varepsilon=0.001$; (b) $\varepsilon=0.0005$. The curve $v=\gamma(u)$ is plotted in black. Red points represent initial conditions corresponding to $u_\infty>1-1/\alpha$, blue dots initial conditions corresponding to $u_\infty<1-1/\alpha$. Note the increase in prediction power as $\varepsilon$ decreases.\label{fig:test_gamma}} 
\end{figure}
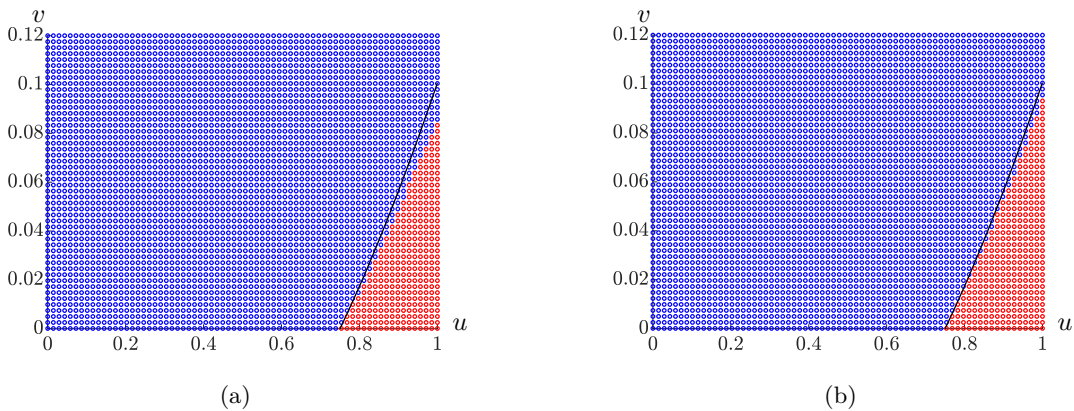

In Figure \ref{fig:homoclinic} we focus on the homoclinic orbits and the heteroclinic cycle corresponding to Figure \ref{fig:sketch5}. Specifically, in Figures \ref{fig:homoc_a} and \ref{fig:homoc_a2}  we illustrate the fact that the homoclinic orbits to $\mathbf{x_0}$ foliate the space between the $u$-axis and the heteroclinic cycle connecting $\mathbf{x_0}$ and $\mathbf{x_1}$. On the other hand, in Figures \ref{fig:homoc_b} and \ref{fig:homoc_d} we analyze how well the curve $v=\gamma(u)$ \eqref{eq:curve_gamma} approximates the heteroclinic orbit evolving on the fast flow from $\mathbf{x_1}$ to $\mathbf{x_0}$. Again, note the increase in prediction power as $\varepsilon$ decreases. Moreover, note that, as demonstrated in Appendix \ref{sec:blow-up_d} (see also Figure \ref{fig:ppK1}), when $2d-\beta<0$ those orbits approach the origin following the $v$-axis (panels (a) and (c)), while when $2d-\beta>0$ from another direction (panels (b) and (d)). 

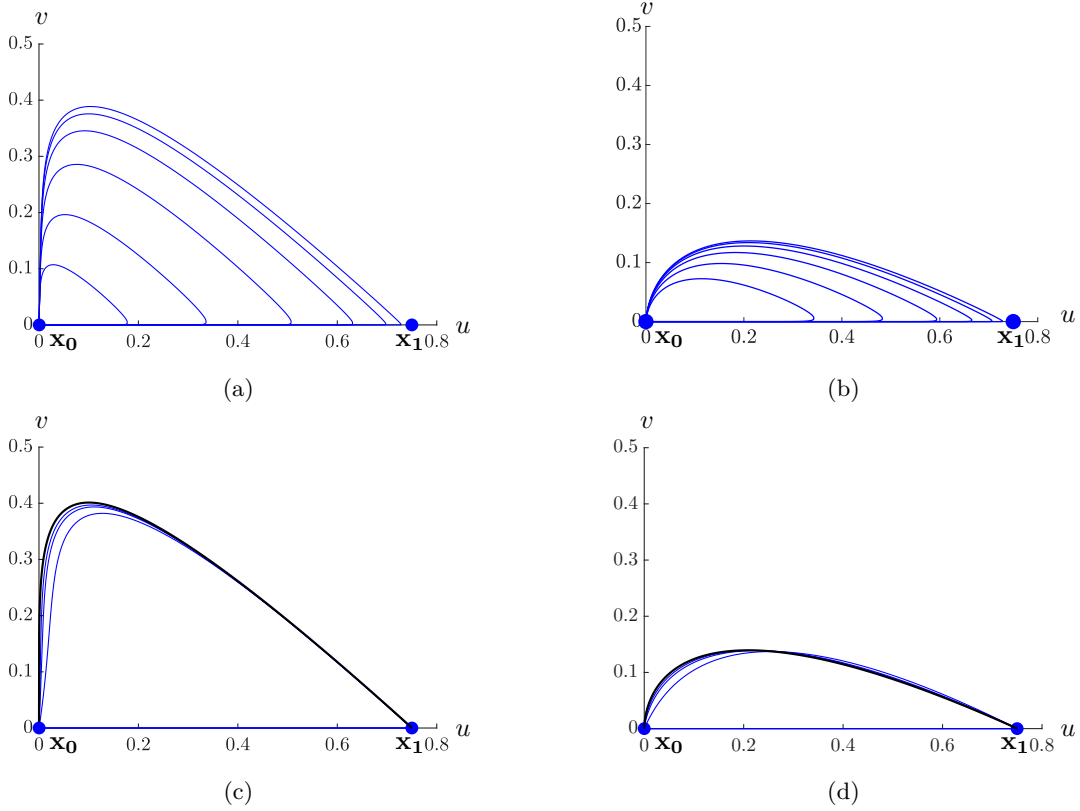
\begin{figure}[h!]
    \centering
\begin{subfigure}{.45\textwidth}
  \centering   
\begin{tikzpicture}
 \node at (0,0) {\includegraphics[width=.8\linewidth]{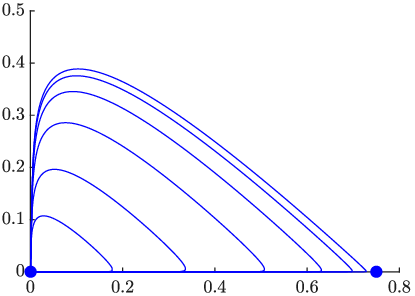}};
\node at (3.1,-1.8) {$u$};
\node at (-2.5,2.3) {$v$};
\node at (2.4,-2) {$\mathbf{x_1}$};
\node at (-2.2,-2) {$\mathbf{x_0}$};
  \end{tikzpicture}
  \caption{} \label{fig:homoc_a} 
\end{subfigure}\hspace{0.5cm}
\begin{subfigure}{.45\textwidth}
  \centering   
\begin{tikzpicture}
 \node at (0,0) {\includegraphics[width=.8\linewidth]{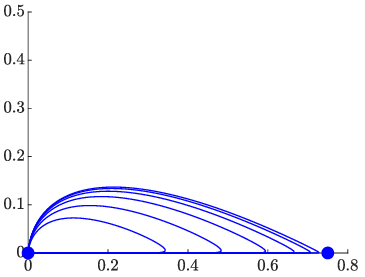}};
\node at (3.1,-1.8) {$u$};
\node at (-2.5,2.3) {$v$};
\node at (2.35,-2.1) {$\mathbf{x_1}$};
\node at (-2.2,-2.1) {$\mathbf{x_0}$};
  \end{tikzpicture}
  \caption{} \label{fig:homoc_a2} 
\end{subfigure}\\
\begin{subfigure}{.45\textwidth}
  \centering
\begin{tikzpicture}
 \node at (0,0) {\includegraphics[width=.8\linewidth]{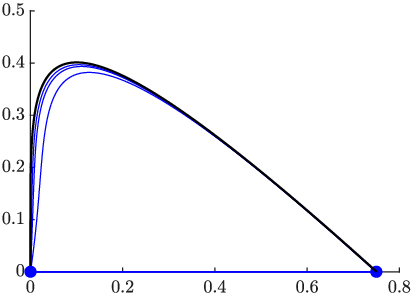}};
\node at (3.1,-1.8) {$u$};
\node at (-2.5,2.3) {$v$};
\node at (2.4,-2) {$\mathbf{x_1}$};\node at (-2.2,-2) {$\mathbf{x_0}$};
  \end{tikzpicture}
  \caption{} \label{fig:homoc_b} 
\end{subfigure}\hspace{0.5cm}
\begin{subfigure}{.45\textwidth}
  \centering
\begin{tikzpicture}
 \node at (0,0) {\includegraphics[width=.8\linewidth]{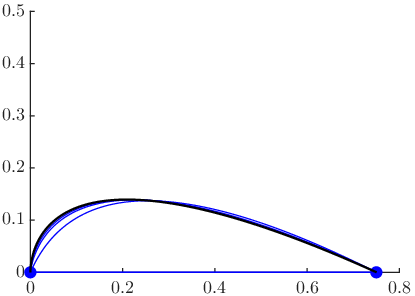}};
\node at (3.1,-1.8) {$u$};
\node at (-2.5,2.3) {$v$};
\node at (2.4,-2) {$\mathbf{x_1}$};\node at (-2.2,-2) {$\mathbf{x_0}$};
  \end{tikzpicture}
  \caption{} \label{fig:homoc_d} 
\end{subfigure}
    \caption{Homoclinic orbits of system \eqref{eq:rescaled}. Values of the parameters: $\alpha=4$, $\theta=0.5$, $\beta=0.5$, and $d$ varying to showcase the two scenarios in Figure \ref{fig:ppK1}; note the qualitatively distinct behavior of the orbits near the $v$-axis, as well as the maximum $v$ values achieved by the orbits (panels (a) and (c) vs. panels (b) and (d)). (a) Multiple homoclinic orbits from $\mathbf{x_0}$ (the origin) to itself, for $\varepsilon=0.005$, $d=0.1$; (b) multiple homoclinic orbits from $\mathbf{x_0}$ (the origin) to itself, for $\varepsilon=0.005$, $d=0.3$; (c) comparison of homoclinic orbits (blue) close to the heteroclinic loop between $\mathbf{x_0}$ and $\mathbf{x_1}$ (blue dots) and the curve $v=\gamma(u)$ (black), recall \eqref{eq:curve_gamma}, for $d=0.1$, $\varepsilon=0.025$, $0.01$  and $0.005$ (innermost to outermost); (d) comparison of homoclinic orbits (blue) close to the heteroclinic loop between $\mathbf{x_0}$ and $\mathbf{x_1}$ (blue dots) and the curve $v=\gamma(u)$ (black), for $d=0.3$, $\varepsilon=0.025$, $0.01$  and $0.005$ (rightmost to leftmost). Note how the approximation improves as $\varepsilon$ decreases. The distance of the orbits and the curve $\gamma$ close to $u=0$ is due to the nullcline $L_1$ not being traversed by the orbits, whereas $\gamma$ ``ignores'' it, as it is a curve that exists in the fast flow (recall that $L_1$ collapses on the $v$-axis as $\varepsilon \to 0$). 
    \label{fig:homoclinic}}
\end{figure}

In Figure \ref{fig:end_eq_grad}, we illustrate the fact that, in the case depicted in Figure \ref{fig:sketch3}, the endemic equilibrium $\mathbf{x_2}$ exists for a $\mathcal{O}(\varepsilon)$-range of values of $\mathcal{R}_0$. Increasing $\beta$, the equilibrium $\mathbf{x_2}$ bifurcates from $\mathbf{x_1}$ as $\mathcal{R}_0=1$ and collapses on $\mathbf{x_0}$ shortly after. In particular, the $u$-coordinate of $\mathbf{x_2}$ travels $\mathcal{O}(1)$-distances for a $\mathcal{O}(\varepsilon)$-variation of $\beta$. Note that, as we remarked in Proposition \ref{prop:EE} and Theorem \ref{teo:EE}, the $v$-coordinate of $\mathbf{x_2}$ is $\mathcal{O}(\varepsilon)$.

\begin{figure}[h!]
\centering
\begin{subfigure}{.45\textwidth}
  \centering
\begin{tikzpicture}
 \node at (0,0) {\includegraphics[width=.85\linewidth]{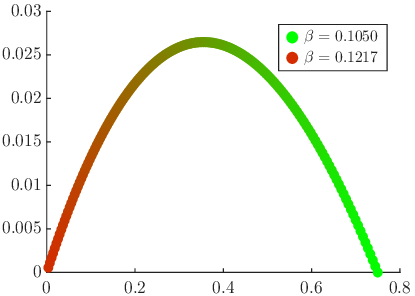}};
\node at (3.1,-1.8) {$u$};
\node at (-2.5,2.4) {$v$};
  \end{tikzpicture}
  \caption{}
  \label{fig:eq1}
\end{subfigure}\hspace{0.5cm}
\begin{subfigure}{.45\textwidth}
    \centering
\begin{tikzpicture}
 \node at (0,0) {\includegraphics[width=.85\linewidth]{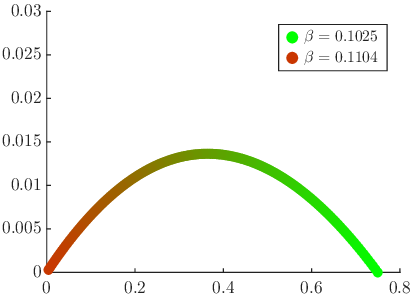}};
\node at (3.1,-1.8) {$u$};
\node at (-2.5,2.4) {$v$};
  \end{tikzpicture}
  \caption{}
  \label{fig:eq2}
\end{subfigure} \\ 
\vspace{.5cm} 
\begin{subfigure}{.45\textwidth}
   \centering
\begin{tikzpicture}
 \node at (0,0) {\includegraphics[width=.85\linewidth]{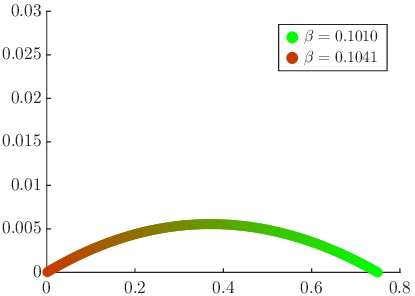}};
\node at (3.1,-1.8) {$u$};
\node at (-2.5,2.4) {$v$};
  \end{tikzpicture}
  \caption{}
  \label{fig:eq3}
\end{subfigure} \hspace{0.5cm}
\begin{subfigure}{.45\textwidth}
   \centering
\begin{tikzpicture}
 \node at (0,0) {\includegraphics[width=.85\linewidth]{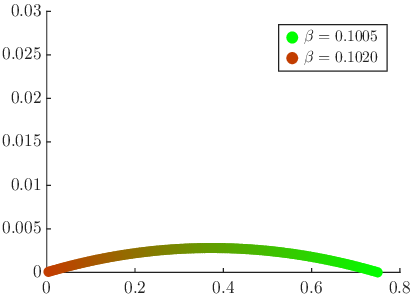}};
\node at (3.1,-1.8) {$u$};
\node at (-2.5,2.4) {$v$};
  \end{tikzpicture}
  \caption{}
  \label{fig:eq4}
\end{subfigure}
\caption{Coordinates of the endemic equilibrium $\mathbf{x_2}$ and corresponding ranges of values of $\beta$ for $\alpha=4$, $\theta=0.5$, $d=0.1$, and (a) $\varepsilon=0.005$; (b) $\varepsilon=0.0025$; (c) $\varepsilon=0.001$; (d) $\varepsilon=0.0005$. Note that the maximum $v$ value achieved and the width of values which $\beta$ assumes while $\mathbf{x_2}$ moves between $\mathbf{x_1}=(0.75,0)$ (bright green dot; $\beta=d+\varepsilon$, i.e., $\mathcal{R}_0=1$) and $\mathbf{x_0}=(0,0)$ (red dot; $\beta=d+\varepsilon \alpha^*$ \eqref{eq:k}) are both $\mathcal{O}(\varepsilon)$. Note that, in line with Proposition \ref{prop:EE}, $\alpha^*\approx\alpha=4$.
\label{fig:end_eq_grad}}\end{figure}

Finally, in Figure \ref{fig:full_dyn}, we again focus on the case depicted in Figure \ref{fig:sketch3}, showcasing how orbits are initially attracted to the equilibrium $\mathbf{x_0}$, travel close to it until they intersect the nullcline $L_2$, after which they increase their $u$-coordinate and are finally attracted to the endemic equilibrium $\mathbf{x_2}$. Note that $\mathbf{x_2}$ is a node; however, as explained in the proof of Theorem \ref{teo:EE}, it could be a focus for other values of the parameters. 

\begin{figure}[h!]
\centering
  \centering
\begin{tikzpicture}
 \node at (0,0) {\includegraphics[width=.7\linewidth]{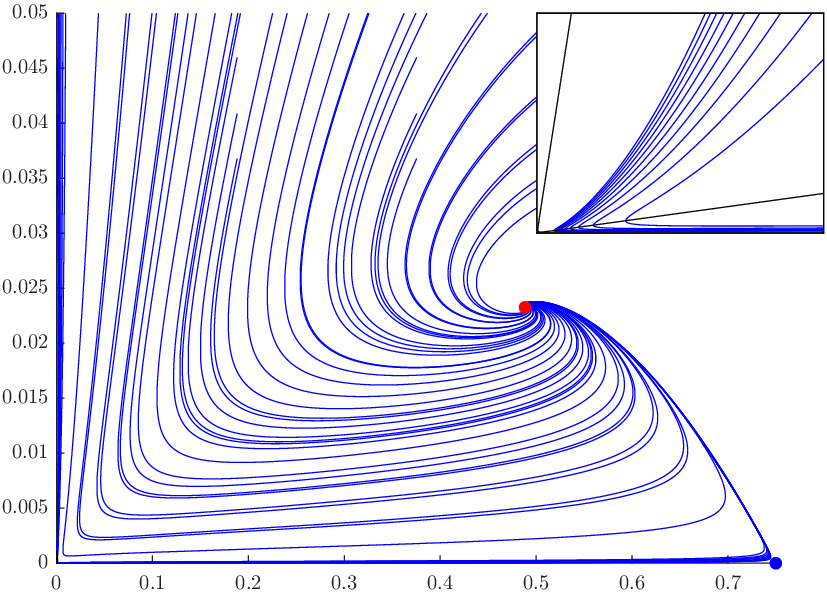}};
\node at (5.4,-3.8) {$u$};
\node at (-5.2,4.4) {$v$};
\node at (4.7,1.6) {$L_2$};
\node at (2.4,3.4) {$L_1$};
\node at (5,-4) {$\mathbf{x_1}$};
\node at (1.5,0.1) {$\mathbf{x_2}$};
  \end{tikzpicture}
\caption{Dynamics of system \eqref{eq:rescaled} for values of the parameters $\alpha=4$, $\theta=0.5$, $\varepsilon=0.005$, $d=0.1$, and $\beta=0.11$, corresponding to $\mathcal{R}_0 \approx 1.0476$ (recall Figure \ref{fig:eq1}). The orbits plotted originate from a grid of distinct points in $\Delta$ \eqref{eq:Delta}. The inset in the top right corner is a zoom-in of the behavior of the orbits as they pass close to the origin $\mathbf{x_0}=(0,0)$, together with the nullclines $L_1$ and $L_2$ (black curves). The red dot represents the EE $\mathbf{x_2}$ and the blue dot represents the DFE $\mathbf{x_1}$. Note the passage of some of the orbits very close to $\mathbf{x_0}$, and subsequently to $\mathbf{x_1}$.
\label{fig:full_dyn}}\end{figure}

\section{Discussion} \label{sec:concl}

In this section, we recap our mathematical results for system \eqref{eq:model} (equivalently, for system \eqref{eq:rescaled}) and we comment on the biological interpretation of them. In particular, we highlight the connections with the DFTD (recall Section \ref{sec:DFTD}).

\subsection{Mathematical results}

In this paper, we introduced a slow-fast version of the model developed by Hwang and Kuang \cite{hwang2003deterministic} to analyze parasite-induced host extinction. Assuming that the birth rate and parasite-independent death rate are significantly lower than the other parameters governing the system, we fully characterized its asymptotic behavior. Our analysis relied on techniques from Geometric Singular Perturbation Theory (GSPT), specifically exploiting the \emph{blow-up} method. 

We showed that the orbits can interact with an infected hosts-free manifold in five different ways, depending on the relationships between the parameters governing the system. The two most interesting cases regard the interaction with a singular point, whose analysis relied on the blow-up technique. Our analysis tackled two main parameter regimes, namely $\beta-d\in\mathcal O(1)$ and $\beta-d\in\mathcal O(\varepsilon)$, and we show that, in both regimes, the local dynamics near extinction ($\mathbf{x}_0$) are radically different, with global consequences. The main advantage of using the blow-up technique is that it has allowed us to describe how the orbits pass through a small neighborhood of such a singular point, especially in a vanishingly small interval of parameters.

In general, our model appears to be very sensitive to the parameter $\beta$, describing the infection rate. Indeed, $\mathcal{O}(\varepsilon)$-variations of $\beta$ lead to the birth or death of the endemic equilibrium ($\mathbf{x_1}$) and to a change in stability of the origin ($\mathbf{x}_0$), which as $\beta$ spans an $\mathcal{O}(\varepsilon)$-large interval switches from being globally attracting to being a saddle.

We performed several numerical simulations to compare the results derived during the analytical analysis in the limit $\varepsilon \to 0$ with the behavior of the orbits corresponding to small but finite values of $\varepsilon$. A good agreement was generally observed for values of $\varepsilon$ on the order of $10^{-3}$. Due to the stiffness of the original system, we introduced a convenient change of variables, bringing it into an equivalent version which allows accurate numerical simulations. 

\subsection{Modelling the spread of devil facial tumor disease}

Our model is able to reproduce the interesting dynamics observed in the spread of the DFTD. Indeed, it predicts abrupt declines in the number of individuals since the decrease in the number of hosts occurs during the fast flow. Moreover, the model can reproduce the hypothesized devils extinction since, for example, the homoclinic orbits converge toward extinction ($\mathbf{x_0}$). Finally, the fact that the disease persists at very low population densities is in line with the tendency of the orbits to interact with the origin ($\mathbf{x_0}$). 

One of our main results concerns the precarious nature of co-existence of the host species and the parasite. Our analysis highlights the potentially lethal effect a parasite could have on the host population -- an effect that is harder to avoid the larger the ratio between the infection and demography timescales is (recall that this ratio is represented by our singular perturbation variable $\varepsilon$, and that the infection does not lead to extinction of the host population only for values of $\mathcal{R}_0-1 \in \mathcal{O}(\varepsilon)$). Again, this is in line with different studies on the DFTD, some of which suggested the possibility of devil extinction in 20-25 years \cite{frequency}, while others that the coexistence between hosts and disease may occur through low-density endemic states \cite{history}. This highlights the precarious situation of the Tasmanian devil. 

These dynamics relate to possible prevention measures: assuming an intervention can lower the value of $\mathcal{R}_0$, this might make a difference between extinction of the host population and extinction of the parasite, within a $\mathcal{O}(\varepsilon)$-range of the parameters involved in the system. We remark on the prediction potential of our GSPT approach; indeed, we are able to determine both the transient and asymptotic dynamics of the system, focusing on either the fast or the slow part of the dynamics. 

The high mortality of the DFTD in few months implies that infected devils might not survive for a second breeding. This limitation in reproducing is captured by our model through the relative fecundity of an infected host $\theta$. Anyway, our mathematical analysis demonstrated that this parameter does not play a key role. Indeed, a reduction of $\theta$ prevents, at most, small increases in the number of hosts. 
 
\appendix

\section{Blow-up analysis} \label{section:blowup}

In order to understand the behavior of the orbits near the singular equilibrium $\mathbf{x_0}$, we use the so-called \emph{blow-up} technique \cite{40,survey,KrupaSzmolyan2001}. The blow-up replaces a singular point by a higher-dimensional manifold (typically a sphere or a cylinder) through a singular change of variables and time rescaling. The technique allows one to recover hyperbolicity on the blown-up space, after which local dynamics can be classified using classical tools of dynamical systems and invariant manifold theory. Most applications have been in two or three dimensions \cite{leslie,kuehn2015multiscale}. Still, extensions to multi-parameter and higher-dimensional systems have been explored in recent works \cite{baumgartner2025multiparameter,jardon2020fast,Jardon-Kojakhmetov_2024,jardon2023network}.

In essence, the blow-up technique allows us to study the behavior of orbits in a small neighborhood of the origin, even though in the original system \eqref{eq:rescaled} the Jacobian evaluated at the corresponding singular equilibrium $\mathbf{x_0}=(0,0)$ is not well defined. For convenience, we recall system \eqref{eq:rescaled}:
\begin{align} \label{eq:repeated}
\begin{aligned}
    \dot{u}&=\varepsilon \alpha(u+\theta v)(1-u-v)-\varepsilon u-\beta \frac{u v}{u+v}, \\[1mm]
    \dot{v}&=\beta \frac{u v}{u+v}-d v-\varepsilon v,
\end{aligned}
\end{align}
which we multiply by $(u+v)$ to obtain the auxiliary system (where, for ease of reading, we avoid the introduction of auxiliary variables and of time reparametrization) 
\begin{align} \label{eq:aux}
\begin{aligned}
    \dot{u}&=\varepsilon \alpha(u+\theta v)(1-u-v)(u+v)-\varepsilon u(u+v)-\beta u v, \\[1mm]
    \dot{v}&=(\beta-d)uv-dv^2-\varepsilon v(u+v).
\end{aligned}
\end{align}

We remark that \eqref{eq:aux} and \eqref{eq:repeated} (i.e., \eqref{eq:rescaled}) are smoothly equivalent away from $\mathbf{x_0}$, which follows from the fact that $(u+v)>0$ for all $(u,v)\in\Delta\backslash\left\{(0,0)\right\}$. The value of $\beta-d$, and whether this quantity is $\mathcal{O}(1)$ or $\mathcal{O}(\varepsilon)$, will play an important role in the dynamics near $\mathbf{x_0}$. We are interested in understanding the behavior of the orbits close to $\mathbf{x_0}$ when: 
\begin{enumerate}
    \item $\alpha>0$ and $\beta-d>0$ with $\beta-d\in\mathcal O(1)$; discussed in Section \ref{sec:blow-up_d}. 
    \item $\alpha>1$ and $\beta-d>0$ with $\beta-d\in\mathcal O(\varepsilon)$; discussed in Section \ref{sec:blow-up_e}.
\end{enumerate}
We note that for the second case, the EE $\mathbf{x_2}$ becomes relevant, see Proposition \ref{prop:EE}.

Setting $\varepsilon=0$ in \eqref{eq:aux}, we obtain the layer equation
\begin{align}\label{eq:aux_eps0}
\begin{aligned}
\dot{u} &= -\beta u v,\\[1mm]
\dot{v} &= (\beta-d)uv-dv^2.
\end{aligned}
\end{align}
Hence, in the first quadrant, the critical manifold of the auxiliary system \eqref{eq:aux} is
\begin{equation}
    \mathcal{C}_0 \coloneqq \{(u,v)\in\mathbb{R}^2_{\ge 0} : v=0\},
\end{equation}
that is, the $u$-axis (including the origin $\mathbf{x_0}$). Since the leading part of \eqref{eq:aux} is purely quadratic, it is straightforward to see that the origin is nilpotent, irrespective of $\varepsilon$.  We shall prove the following.

\begin{proposition}\label{thm:blow} 
Consider \eqref{eq:rescaled}. Then, in the biologically relevant region $\Delta$ \eqref{eq:Delta}, the singular equilibrium $\mathbf{x_0}$ has the following properties:
\begin{itemize}
    \item If $\alpha<1$ and $0<\beta-d\in\mathcal O(1)$, then $\mathbf{x_0}$ is a topological sink.
    \item If $\alpha>1$ and $0<\beta-d\in\mathcal O(1)$, then $\mathbf{x_0}$ has a sector foliated by homoclinic orbits to itself. 
    \item  If $\alpha>1$ and $\beta=d+\varepsilon k$, with $1<k<\alpha^*$, then $\mathbf{x_0}$ is a topological saddle with the $v$-axis attracting and the $u$-axis repelling (up to $O(\ve^2)$-corrections).
\end{itemize}
\end{proposition}

\begin{proof}
    The details of the first two items are given in Section \ref{sec:blow-up_d}, while those for the third item are provided in Section \ref{sec:blow-up_e}. We emphasize that our focus is in describing the flow of \eqref{eq:rescaled} in a small neighborhood of the origin restricted to the first quadrant. The general strategy is not to explicitly study \eqref{eq:rescaled} but auxiliary systems, in particular \eqref{eq:aux}, that are smoothly equivalent to \eqref{eq:rescaled} away from the origin. The main techniques we use are nowadays standard tools of GSPT \cite{40,survey}.
\end{proof}

\subsection{Desingularization for $0<\beta-d\in\mathcal O(1)$}\label{sec:blow-up_d}

Here, we start by assuming $0<\beta-d \in \mathcal{O}(1)$ and $\alpha>0$.  We blow up the origin of \eqref{eq:aux} without considering $\varepsilon$ as a singular parameter. We propose the homogeneous blow-up $\Phi:\mathbb S^1\times\mathbb R_{\geq0}\to\mathbb R^2$ given by
\begin{equation}
     u=r\bar u, \qquad v=r\bar v,
\end{equation}
where $\bar u^2+\bar v^2=1$ and $r\ge0$. The charts that are relevant for us are
\begin{equation}\label{eq:b0K1}
    K_1:\quad u=r_1u_1,\;\;v=r_1
\end{equation}
and 
\begin{equation}\label{eq:b0K2}
    K_2:\quad u=r_2,\;\;v=r_2v_2.
\end{equation}
Figure \ref{fig:charts} shows a sketch of the above-mentioned charts.
\begin{figure}
    \centering
    \begin{tikzpicture}
        \node at (0,0){\includegraphics[scale=1]{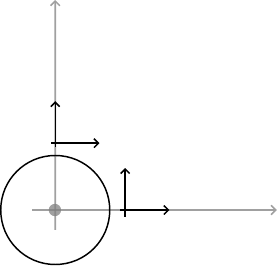}};
        \node[gray] at (-1.4,2.4) {$v$};
        \node[gray] at (2.5,-1.3) {$u$};
        \node at (0.65,-1.1) {$r_2$};
        \node at (0.0,-.5) {$v_2$};
        \node at (-0.55,.05) {$u_1$};
        \node at (-1.1,.5) {$r_1$};
    \end{tikzpicture}
    
    \caption{A sketch of the charts $K_1$ and $K_2$. The origin is blown up to a $1$-sphere (black). Each chart covers a portion of the blow-up sphere, and the local coordinates used in our analysis are also indicated. Chart $K_1$, with local coordinates $(r_1,u_1)$, ``looks at the sphere from above'', while chart $K_2$, with local coordinates $(r_2,v_2)$, ``looks at the sphere from the right''.}
    \label{fig:charts}
\end{figure}
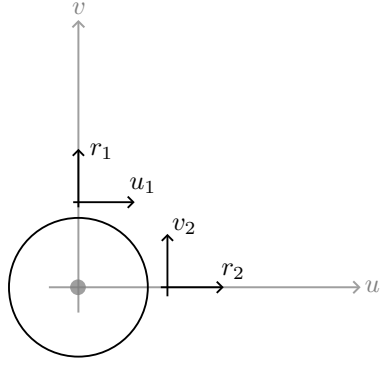

\subsubsection*{Analysis in the chart $K_1$}

In the first chart, the corresponding desingularized vector field \cite{survey}, after division by $r_1$, reads as: 
\begin{align}\label{eq:K1}
\begin{aligned}
\dot r_1
&= -\,r_1\big(d+\varepsilon + (d-\beta+\varepsilon)u_1\big),\\
\dot u_1
&= (1+u_1)
\Big(
\alpha\varepsilon\theta
+ (d-\beta+\alpha\varepsilon)u_1
- \alpha\varepsilon r_1(1+u_1)(\theta+u_1)
\Big).
\end{aligned}
\end{align}
So, the reduced flow on the blown-up sphere $\{r_1=0\}$ is
\begin{equation}\label{eq:K1_reduced_factored}
\begin{split}
    \dot r_1 &= 0,\\
\dot u_1 &= (1+u_1)\big(\alpha\varepsilon\theta + (d-\beta+\alpha\varepsilon)u_1\big).
\end{split}
\end{equation}
Hence the equilibria on $\{r_1=0\}$ satisfy
$(1+u_1)\big(\alpha\varepsilon\theta + (d-\beta+\alpha\varepsilon)u_1\big) = 0$,
which has two roots: 
\begin{equation}\label{eq:K1_equilibria_exact}
u_1^{(0)}(\varepsilon)
= \frac{\varepsilon\alpha\theta}{\beta-d-\alpha\varepsilon}, \qquad u_1^{(1)}(\varepsilon) = -1.
\end{equation}
The equilibrium $u_1^{(1)}$ has a negative value, which is not in the biologically relevant region, hence we only focus on $u_1^{(0)} \ge 0$. 
For $\varepsilon$ small, we can expand $u_1^{(0)}(\varepsilon)$ as 
\begin{equation}\label{eq:K1_equilibria_expansion}
u_1^{(0)}(\varepsilon)
= \frac{\alpha\theta}{\beta-d}\,\varepsilon
+ \mathcal{O}(\varepsilon^2).
\end{equation}
The Jacobian of \eqref{eq:K1} evaluated at $(0,u_1^{(0)})$ reads as
\begin{equation}\label{eq:K1_J_eq_u0}
J_1(0,u_1^{(0)})
=
\begin{bmatrix}
-\big(d+\varepsilon + (d-\beta+\varepsilon)u_1^{(0)}\big) & 0\\[1mm]
-\,\alpha\varepsilon\big(1+u_1^{(0)}\big)^2\big(\theta+u_1^{(0)}\big)
&
\big(1+u_1^{(0)}\big)\,(d-\beta+\alpha\varepsilon)
\end{bmatrix},
\end{equation}
therefore its eigenvalues are
\begin{equation} \label{eq:lambda_r}
\lambda_r(\varepsilon)
= -\big(d+\varepsilon + (d+\varepsilon-\beta)u_1^{(0)}(\varepsilon)\big),
\qquad
\lambda_u(\varepsilon)
= \big(1+u_1^{(0)}(\varepsilon)\big)\,(d+\alpha\varepsilon-\beta),
\end{equation}
which have the expansion
\begin{equation}
\lambda_r(\varepsilon)
= -d - (1-\alpha\theta)\,\varepsilon + \mathcal{O}(\varepsilon^2), \qquad
\lambda_u(\varepsilon)
= -(\beta-d) + \alpha(1-\theta)\,\varepsilon + \mathcal{O}(\varepsilon^2).
\end{equation}
Note that, for $\varepsilon$ small enough, $\lambda_r(\varepsilon), \lambda_u(\varepsilon)<0$. Therefore, the equilibrium $p_1\coloneqq(r_1,u_1)=(0,u_1^{(0)}(\varepsilon))$ is a hyperbolic sink in the chart $K_1$. The corresponding eigenvectors are 
\begin{equation}
\nu_r(\varepsilon)=
    \begin{bmatrix}
        \frac{2 d-\beta}{\alpha}+\mathcal O(\varepsilon) \\ \varepsilon\theta
    \end{bmatrix}, \quad \nu_u(\varepsilon)=
    \begin{bmatrix}
        0 \\ 1
    \end{bmatrix}. \quad
\end{equation}

With the analysis presented above, we conclude that the local phase-portrait close to the equilibrium point $p_1$ is as shown in Figure \ref{fig:ppK1}.
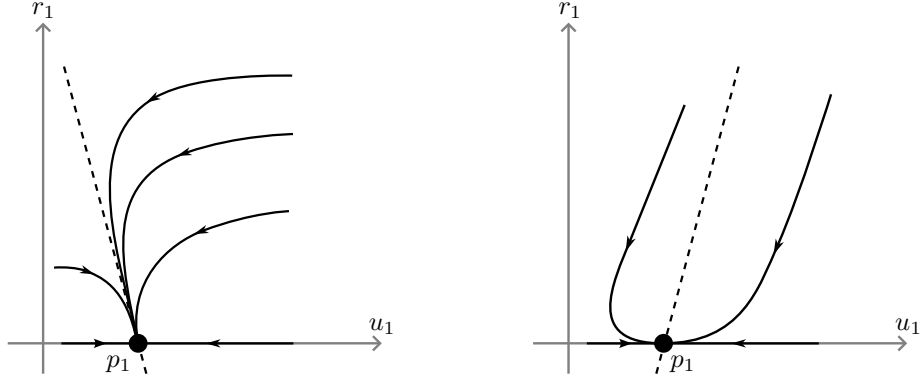
\begin{figure}[htbp]
    \centering
    \begin{tikzpicture}
        \node at (0,0){\includegraphics[scale=1.2]{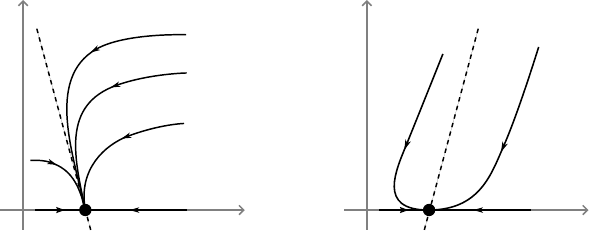}};
        \node at (-1,-1.7){$u_1$};
        \node at (-5.5,2.5){$r_1$};
        \node at (6,-1.7){$u_1$};
        \node at (1.5,2.5){$r_1$};
        \node at (-4.5,-2.2){$p_1$};
        \node at (3,-2.2){$p_1$};
    \end{tikzpicture}
    \caption{Phase-portrait of \eqref{eq:K1} in a small neighborhood of the equilibrium point $p_1$. The dashed line indicates the eigenvector $\nu_r$. In the left picture, we show the case when $2d-\beta<0$ for which the stronger contraction direction is $\nu_u$. On the right, we show the case when $2d-\beta>0$, for which the stronger contraction direction is $\nu_r$ instead. In other words, the main difference between the two cases is that for $2d-\beta<0$, the trajectories approach $p_1$ along the eigendirection $\nu_r$. Recall Figure \ref{fig:homoclinic} for a comparison of the two cases portrayed here in the original system \eqref{eq:rescaled}.
    \label{fig:ppK1}}
\end{figure}

\subsubsection*{Analysis in the chart $K_2$}

In the second chart, the desingularized system, after division by $r_2$, reads as
\begin{align}\label{eq:K2}
\begin{aligned}
\dot r_2
&= r_2\Big(
   -\alpha\varepsilon (1+v_2)(1+\theta v_2)(r_2 v_2 + r_2 - 1)
   - \beta v_2 - \varepsilon(1+v_2)
\Big),\\[1mm]
\dot v_2
&= v_2\Big(
   \alpha\varepsilon (1+v_2)(1+\theta v_2)(r_2 v_2 + r_2 - 1)
   + \beta v_2 + \beta - d(1+v_2)
\Big).
\end{aligned}
\end{align}
On the blown-up sphere $\{r_2=0\}$, the reduced flow is:
\begin{equation}\label{eq:K2_reduced}
\begin{split}
    \dot r_2 &= 0,\\
\dot v_2 &= -\,v_2(1+v_2)\big(\alpha\varepsilon\theta v_2 + \alpha\varepsilon - \beta + d\big).
\end{split}
\end{equation}
The equilibria of the reduced flow on $\{r_2=0\}$ are therefore given by the three roots
\begin{equation}\label{eq:K2_equilibria}
v_2^{(0)}(\varepsilon) = 0,\qquad
v_2^{(1)}(\varepsilon) = -1,\qquad
v_2^{(2)}(\varepsilon)
= \frac{\beta - d - \alpha\varepsilon}{\alpha\varepsilon\theta}.
\end{equation}
Of the three equilibria, we shall focus only on the first. The second (corresponding to $v_2^{(1)}(\varepsilon)$) lies in the biologically irrelevant region. The third (corresponding to $v_2^{(2)}$) coincides with the equilibrium $u_1^{(0)}$ in chart $K_1$. Indeed, from \eqref{eq:K1_equilibria_exact} one readily checks that
\[
u_1^{(0)}(\varepsilon)\,v_2^{(2)}(\varepsilon) = 1,
\]
which is precisely the transition relation between $K_1$ and $K_2$ in the overlapping region. If $\theta=0$, then the equilibrium corresponding to $v_2^{(2)}(\varepsilon)$ is only visible in chart $K_1$.

To determine the local structure near $v_2=0$, we compute the Jacobian of \eqref{eq:K2} at $(r_2,v_2)=(0,0)$, which gives
\begin{equation}\label{eq:K2_J_00}
J_2(0,0)
=
\begin{bmatrix}
\varepsilon(\alpha-1) & 0\\[1mm]
0 &  \beta - d-\alpha\varepsilon
\end{bmatrix}.
\end{equation}
The eigenvalues at $p_2\coloneqq(r_2,v_2)=(0,0)$ are therefore 
\begin{equation} \label{eq:lambdav}
\lambda_r(\varepsilon) = \varepsilon(\alpha-1)=\varepsilon \mathcal{R}_d,\qquad
\lambda_v(\varepsilon) = \beta - d - \alpha\varepsilon,
\end{equation}
with corresponding eigenvectors $[1,0]^\top$ and $[0,1]^\top$. It is straightforward to observe that: 
\begin{itemize}
    \item For $\alpha>1$, we have that $\lambda_r(\varepsilon), \lambda_v(\varepsilon)>0$; $p_2$ is a hyperbolic source.
    \item For $0<\alpha<1$, we have that $\lambda_r(\varepsilon)<0$ and $\lambda_v(\varepsilon)>0$; $p_2$ is a hyperbolic saddle.
\end{itemize}
In addition, as $\varepsilon\to0$, the $r_2$-axis (i.e., $\left\{ v_2=0\right\}$) becomes a line of equilibria corresponding to the critical manifold $\mathcal{C}_0$.

From \eqref{eq:K2}, we can write
\begin{equation}
\dot r_2 = r_2 A(r_2,v_2;\varepsilon),
\end{equation}
where
\begin{equation}
A(r_2,v_2;\varepsilon)
= -\alpha\varepsilon (1+v_2)(1+\theta v_2)(r_2 v_2 + r_2 - 1)
   - \beta v_2 - \varepsilon(1+v_2).
\end{equation}
For $r_2$ sufficiently small we have $r_2v_2 + r_2 - 1 = -1 + \mathcal{O}(r_2)$, hence
\begin{equation}\label{eq:A_expansion}
A(r_2,v_2;\varepsilon)
= -\beta v_2
 + \varepsilon (1+v_2)\big(\alpha(1+\theta v_2)-1\big)
 + \mathcal{O}(\varepsilon r_2).
\end{equation}
Let $\delta,V\in\mathbb R$ such that $0<\delta<V$. Since the leading term in \eqref{eq:A_expansion} is $-\beta v_2\le -\beta\delta<0$, there exists $\varepsilon_0>0$ such that, for all $0<\varepsilon<\varepsilon_0$ and all $r_2$ sufficiently small,
\begin{equation}
A(r_2,v_2;\varepsilon) < 0 \qquad \text{for } v_2\in[\delta,V].
\end{equation}
Consequently,
\begin{equation}
\dot r_2 = r_2 A(r_2,v_2;\varepsilon) < 0
\quad\text{for }r_2>0,\ v_2\in[\delta,V],
\end{equation}
that is, for $\varepsilon>0$ sufficiently small the flow in chart $K_2$ contracts towards the $v_2$-axis in any strip
$\{v_2\ge\delta>0\}$ of the first quadrant.
We conclude that the phase-portrait for $\varepsilon>0$ close to the equilibrium $p_2$ is as sketched in Figure \ref{fig:ppK2}.

\begin{figure}[htbp]
    \centering
    \begin{tikzpicture}
        \node at (0,0){\includegraphics[scale=1.2]{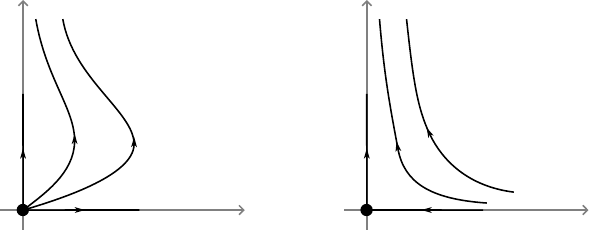}};
        \node at (-1,-1.7){$r_2$};
        \node at (-5.5,2.5){$v_2$};
        \node at (6,-1.7){$r_2$};
        \node at (1.5,2.5){$v_2$};
        \node at (-5.2,-2.2){$p_2$};
        \node at (1.7,-2.2){$p_2$};
    \end{tikzpicture}
    \caption{Phase-portrait of \eqref{eq:K2} in a small neighborhood of the equilibrium point $p_2$. The left picture corresponds to $\alpha>1$ while the right picture corresponds to $0<\alpha<1$.}
    \label{fig:ppK2}
\end{figure}
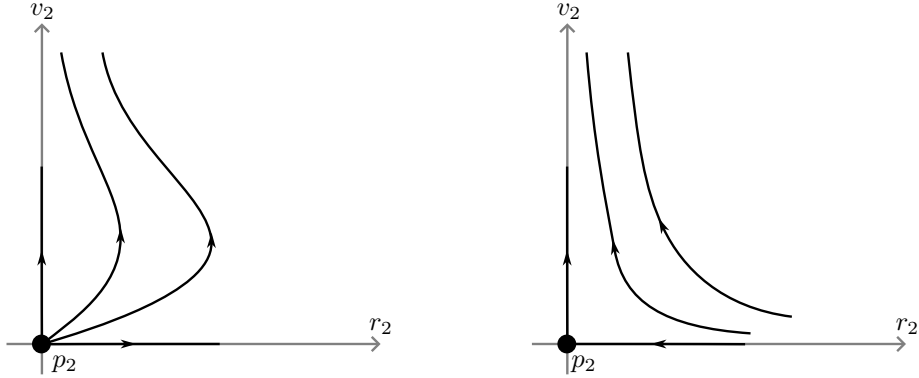

\subsubsection*{Transition between the charts and Summary}

The dynamics in the charts $K_1$ and $K_2$ are matched via the transition maps
$\kappa_{1\to2}(r_1,u_1)=(r_1u_1,\tfrac{1}{u_1})=(r_2,v_2)$ and
$\kappa_{2\to1}(r_2,v_2)=(r_2v_2,\tfrac{1}{v_2})=(r_1,u_1)$ obtained from
\eqref{eq:b0K1} and \eqref{eq:b0K2}. Fix small constants $\delta, \rho_0 > 0$
and define the cross-sections
\begin{equation}\label{eq:cross_sections_A1}
  \Sigma_2^{\mathrm{en}}:=\left\{(r_1, u_1): u_1=1/\delta,\;
    0 \leq r_1 \leq \rho_0\delta\right\}, \quad
  \Sigma_2^{\mathrm{ex}}:=\left\{(r_2, v_2): v_2=\delta,\;
    0 \leq r_2 \leq \rho_0\right\},
\end{equation}
which are identified under $\kappa_{1 \rightarrow 2}$: if
$(r_1, u_1) \in \Sigma_2^{\mathrm{en}}$, then
$\kappa_{1\to2}(r_1, 1/\delta) = (r_1/\delta, \delta)$, and the condition
$0 \le r_1 \le \rho_0\delta$ maps to $0 \le r_2 \le \rho_0$, consistent with
$\Sigma_2^{\mathrm{ex}}$. The notation indicates that orbits leave chart $K_2$
through $\Sigma_2^{\mathrm{ex}}$ and enter chart $K_1$ through
$\Sigma_2^{\mathrm{en}}$. We now describe the passage through each chart.

\medskip 

\noindent\textbf{Passage through chart $K_2$.}
We distinguish two cases according to the sign of $\alpha - 1$.

\begin{itemize}
  \item \textbf{Case $\alpha>1$.} The equilibrium $p_2=(0,0)$ is a hyperbolic
    source. On the blow-up sphere $\{r_2=0\}$, the reduced flow satisfies
    $\dot{v}_2>0$ for $0<v_2<v_2^{(2)}(\varepsilon)$, so orbits starting near
    $p_2$ with $v_2>0$ reach
    $\Sigma_2^{\mathrm{ex}} \cap \{r_2=0\}$ in finite time $T_0 > 0$. It remains to control $r_2$ during this passage for orbits with $r_2 > 0$.
    Since $p_2$ is a source, $r_2$ initially \emph{increases} near $p_2$
    (indeed, $A(0,0;\varepsilon) = \varepsilon(\alpha - 1) > 0$). However, by
    continuous dependence on initial conditions applied to the finite-time
    passage on $\{r_2 = 0\}$, for any $\eta > 0$ there exists $\rho_0 > 0$
    such that orbits starting in a $\mathcal{O}(\rho_0)$-neighborhood of $p_2$ with
    $r_2 > 0$ reach the section $\{v_2 = \delta\}$ with $r_2 \le \eta$. In
    particular, choosing $\rho_0$ sufficiently small, these orbits reach
    $\Sigma_2^{\mathrm{ex}}$.

    Once $v_2 \ge \delta$, the estimate $A(r_2, v_2; \varepsilon) < 0$ (see
    \eqref{eq:A_expansion}) ensures $\dot{r}_2 = r_2 A < 0$ for $r_2 > 0$, so
    $r_2$ is strictly decreasing beyond $\Sigma_2^{\mathrm{ex}}$.

  \item \textbf{Case $0<\alpha<1$.} The equilibrium $p_2$ is a hyperbolic
    saddle with $\lambda_r=\varepsilon(\alpha-1)<0$ and
    $\lambda_v=\beta-d-\alpha\varepsilon>0$. On the sphere $\{r_2 = 0\}$, the
    flow in the $v_2$-direction is the same as for $\alpha > 1$, so orbits on the
    sphere reach $\Sigma_2^{\mathrm{ex}} \cap \{r_2 = 0\}$ in finite time.

    For $r_2 > 0$, orbits approaching $p_2$ along the critical manifold
    ($v_2 \approx 0$) are driven toward $\{r_2 = 0\}$ by the attracting
    eigenvalue $\lambda_r < 0$. By the stable manifold theorem at $p_2$, such
    orbits enter an $\mathcal{O}(\rho_0)$-neighborhood of the sphere and are
    then carried along it by the unstable $v_2$-dynamics. The same continuous
    dependence argument as above shows they reach $\Sigma_2^{\mathrm{ex}}$
    with $r_2$ small.
\end{itemize}

\medskip
\noindent\textbf{Passage through chart $K_1$.}
Via $\kappa_{2 \rightarrow 1}$, orbits arriving at $\Sigma_2^{\mathrm{ex}}$
enter $K_1$ through $\Sigma_2^{\mathrm{en}}$ with $u_1=1/\delta$ and
$0 \leq r_1 \leq \rho_0 \delta$ (after possibly relabeling $\rho_0$). On the
sphere $\{r_1=0\}$, since $d-\beta+\varepsilon\alpha <0$, the reduced flow
\eqref{eq:K1_reduced_factored} satisfies $\dot{u}_1<0$ for
$u_1>u_1^{(0)}(\varepsilon)$, so orbits on the sphere flow monotonically from
$u_1=1/\delta$ toward $p_1=\left(0, u_1^{(0)}(\varepsilon)\right)$.

Since $p_1$ is a hyperbolic sink, there exists a neighborhood $\mathcal{U}$ of
$p_1$ in $\{r_1 \geq 0,\, u_1 \geq 0\}$ such that all orbits entering
$\mathcal{U}$ converge to $p_1$. It remains to show that orbits entering
through $\Sigma_2^{\mathrm{en}}$ reach $\mathcal{U}$. From
\eqref{eq:K1},
$\dot{r}_1 = -r_1(d + \varepsilon + (d - \beta + \varepsilon)u_1)$. For
$\varepsilon$ small, the factor
$d + \varepsilon + (d - \beta + \varepsilon)u_1$ is equivalent to
$d - (\beta - d)u_1 + \mathcal{O}(\varepsilon)$. Since $u_1$ decreases
monotonically from $1/\delta$ toward
$u_1^{(0)}(\varepsilon) = \mathcal{O}(\varepsilon)$, there are two regimes.
If $\beta < 2d$, the factor is positive for all $u_1 \ge 0$, so
$\dot{r}_1 < 0$ throughout and $r_1$ is non-increasing. If $\beta \ge 2d$,
the factor could become negative for large $u_1$, but this only occurs for
$u_1 > d/(\beta - d) \in \mathcal{O}(1)$. During the (finite-time) passage
through the region of large $u_1$, the growth of $r_1$ is bounded, and once
$u_1$ decreases below $d/(\beta - d)$, $r_1$ begins to decrease. By choosing
$\rho_0$ small enough, the total excursion of $r_1$ remains within
$\mathcal{U}$, and the orbit converges to $p_1$.

\medskip
\noindent\textbf{Summary.} Combining the two charts:
\begin{itemize}
  \item For $\alpha > 1$: orbits leaving $p_2$ (except along the invariant
    $r_2$-axis) are carried through $K_2$ to $\Sigma_2^{\mathrm{ex}}$,
    transition to $K_1$, and converge to $p_1$.
  \item For $0 < \alpha < 1$: orbits approaching $p_2$ along the critical
    manifold ($v_2 \approx 0$) are attracted toward the sphere, redirected
    along it through $K_2$ to $\Sigma_2^{\mathrm{ex}}$, and converge to $p_1$
    in $K_1$.
\end{itemize}

\medskip
\noindent\textbf{Blow-down and approach geometry.}
Blowing down via $\Phi^{-1}$, the equilibrium
$p_1=\left(0, u_1^{(0)}(\varepsilon)\right)$ with $u_1^{(0)}(\varepsilon)$
given in \eqref{eq:K1_equilibria_exact} determines the approach to the origin
in the original coordinates. Since $(u, v)= (r_1 u_1, r_1)$, orbits converging
to $p_1$ approach the origin with slope
$v/u = 1/u_1^{(0)}(\varepsilon) \in \mathcal{O}(1/\varepsilon)$, i.e., nearly
tangent to the $v$-axis. The strong/weak stable manifold structure at $p_1$
determines the precise approach geometry in the two subcases
$2d-\beta \lessgtr 0$, as depicted in Figure~\ref{fig:ppK1}: when
$2d - \beta < 0$, the eigendirection $\nu_u$ contracts faster and orbits
approach $p_1$ tangent to $\nu_r$; when $2d - \beta > 0$, the roles are
reversed.

With this analysis, one shows that the flow of \eqref{eq:aux} close to the
origin, for $0<\beta-d \in \mathcal{O}(1)$, is as shown in
Figure~\ref{fig:blow-up1}, proving the first two items of
Proposition~\ref{thm:blow}. We remark that the assumption
$\beta - d \in \mathcal{O}(1)$ can be relaxed to $\beta > d + \varepsilon\alpha^*$
(in line with Section~\ref{sec:R01}), since the analysis only requires
$\beta - d > 0$ together with the hyperbolicity of $p_1$ and $p_2$, both of
which persist as long as $\beta - d - \alpha\varepsilon > 0$, i.e.,
$k > \alpha^*$ up to $\mathcal{O}(\varepsilon)$-corrections. The transition to
the regime $\beta - d \in \mathcal{O}(\varepsilon)$ is discussed in the
forthcoming section.

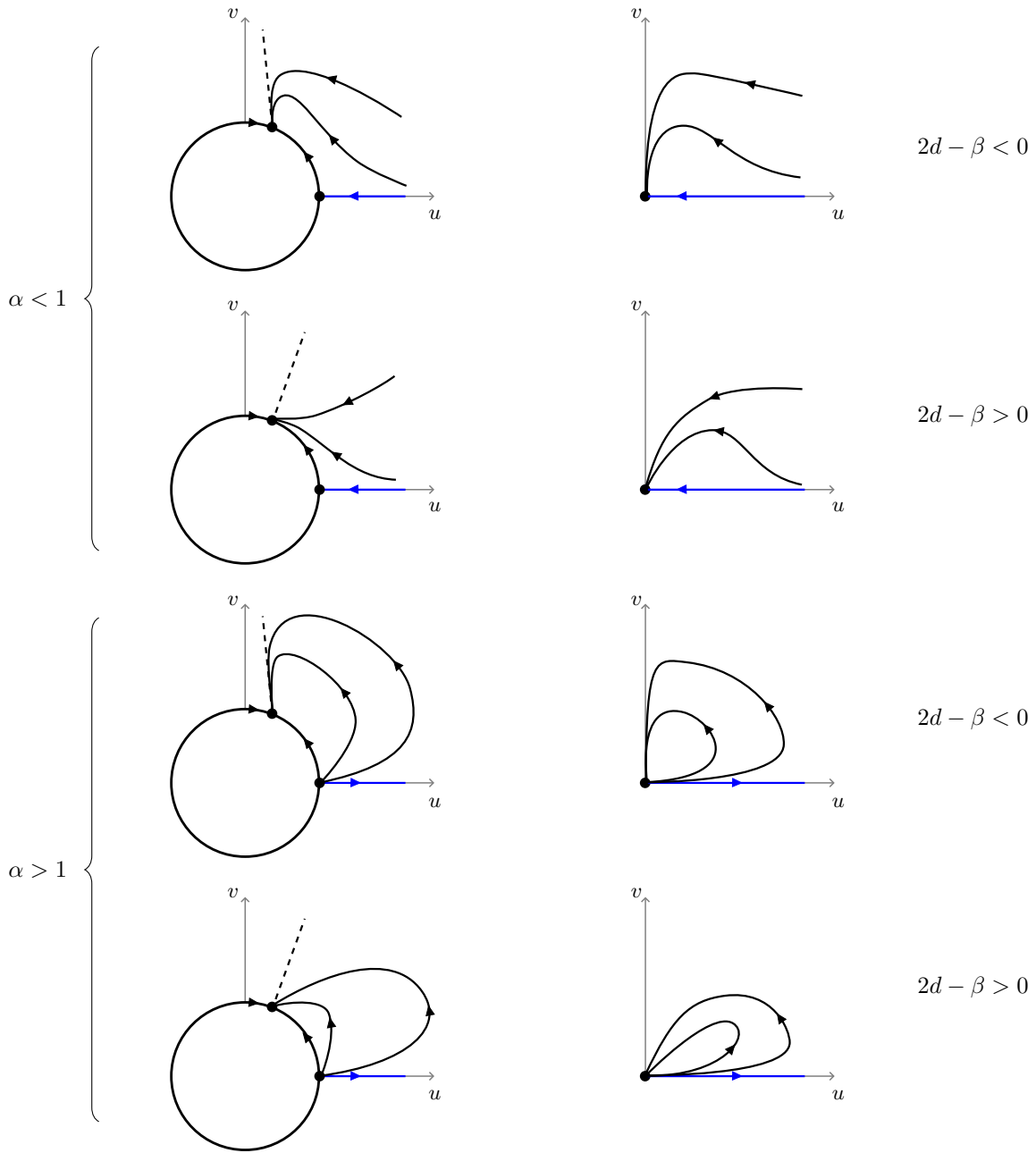
\begin{figure}[htbp]
    \centering
    \begin{tikzpicture}
        \node at (0,0){
        \includegraphics[scale=0.75]{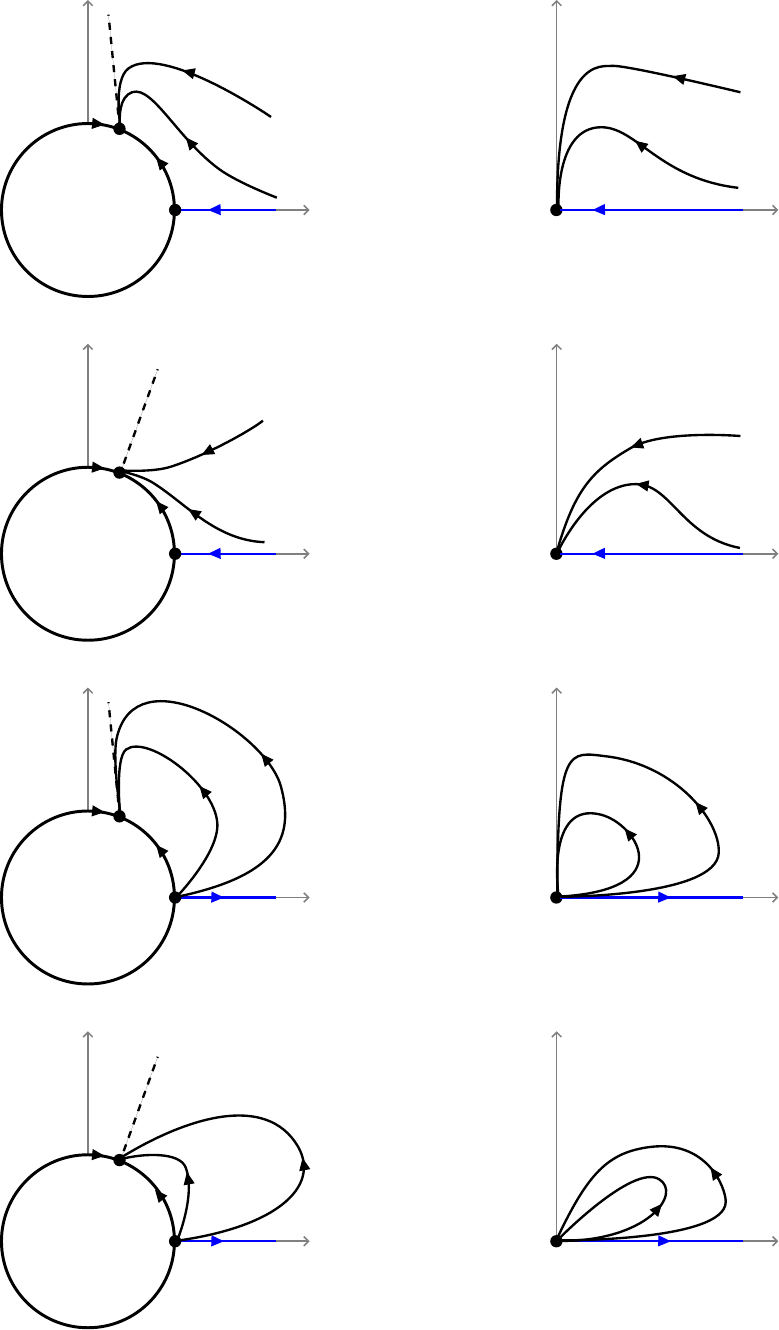}
        };
        \draw [decorate, decoration={brace, mirror,amplitude=6pt}]
  (-6,8) --++(0,-7.5)
  node[midway, left=10pt] {$\alpha<1$};
  \draw [decorate, decoration={brace, mirror,amplitude=6pt}]
  (-6,-0.5) --++(0,-7.5)
  node[midway, left=10pt] {$\alpha>1$};

\node at (-2.5,7){
\begin{tikzpicture}
    \node at (5,0) {\small$u$};
    \node at (2,3) {\small$v$};
\end{tikzpicture}
};

\node at (3.5,7){
\begin{tikzpicture}
    \node at (5,0) {\small$u$};
    \node at (2,3) {\small$v$};
\end{tikzpicture}
};

\node at (-2.5,2.65){
\begin{tikzpicture}
    \node at (5,0) {\small$u$};
    \node at (2,3) {\small$v$};
\end{tikzpicture}
};

\node at (3.5,2.65){
\begin{tikzpicture}
    \node at (5,0) {\small$u$};
    \node at (2,3) {\small$v$};
\end{tikzpicture}
};

\node at (-2.5,-1.75){
\begin{tikzpicture}
    \node at (5,0) {\small$u$};
    \node at (2,3) {\small$v$};
\end{tikzpicture}
};

\node at (3.5,-1.75){
\begin{tikzpicture}
    \node at (5,0) {\small$u$};
    \node at (2,3) {\small$v$};
\end{tikzpicture}
};

\node at (-2.5,-6.1){
\begin{tikzpicture}
    \node at (5,0) {\small$u$};
    \node at (2,3) {\small$v$};
\end{tikzpicture}
};

\node at (3.5,-6.1){
\begin{tikzpicture}
    \node at (5,0) {\small$u$};
    \node at (2,3) {\small$v$};
\end{tikzpicture}
};

  \node at (7,6.5){$2d-\beta<0$};
  \node at (7,2.5){$2d-\beta>0$};
  \node at (7,-2){$2d-\beta<0$};
  \node at (7,-6){$2d-\beta>0$};
    \end{tikzpicture}
    
     \caption{Qualitative representation of the dynamics of \eqref{eq:aux} close to the origin, for $0<\beta-d\in\mathcal O(1)$ and under different parameter regimes as indicated. On the left we show the full blown-up dynamics in the first quadrant, compare with Figures \ref{fig:ppK1} and \ref{fig:ppK2}. On the right we depict the corresponding blown-down dynamics, which for the case $\alpha>1$ are showcased in Figure \ref{fig:homoclinic} as simulations of the original system \eqref{eq:rescaled}. 
    \label{fig:blow-up1}}
\end{figure}

\newpage
\subsection{Desingularization for $0<\beta-d\in\mathcal O(\varepsilon)$}\label{sec:blow-up_e}

In this section, we assume that $\alpha>1$ and  $\beta=d+\varepsilon k$ with $1<k<\alpha^*$, implying the existence of the EE $\mathbf{x_2}$; recall Proposition \ref{prop:EE} and Theorem \ref{teo:EE}.

\begin{remark}
    Recall that, for $\ve>0$ sufficiently small and $\theta>0$, we have $\alpha<\alpha^*$. As we will see in the analysis below, the higher order terms in $\alpha^*=\alpha+\mathcal O(\ve)$ have a regular perturbation role.
\end{remark}

Then, \eqref{eq:aux} is rewritten as
\begin{equation}\label{eq:aux_k_split}
\begin{aligned}
\dot{u} &= - d\,u v
      + \varepsilon\Big(\alpha(u+\theta v)(1-u-v)(u+v) - u(u+v) - k\,u v\Big),\\[1mm]
\dot{v} &= - d\,v^2
      + \varepsilon\big(k\,u v - v(u+v)\big),\\[1mm]
      \dot{\varepsilon}&=0,
\end{aligned}
\end{equation}
where we have also extended the system with the usual trivial equation $\dot{\varepsilon}=0$. Note that for $\varepsilon=0$ the critical manifold $\mathcal{C}_0$ is nilpotent. To resolve this degeneracy, we perform a preliminary rescaling, namely
\begin{equation}\label{eq:preblow}
    u = u, \qquad v = \varepsilon x,
\end{equation}
where $x$ is a new variable. Then, the desingularized vector field reads as
\begin{equation}\label{eq:preblow_desing_full}
\begin{aligned}
\dot u
&= - d\,u x
   + \alpha\big(u+\theta\varepsilon x\big)\big(1-u-\varepsilon x\big)\big(u+\varepsilon x\big)
   - u\big(u+\varepsilon x\big)
   - k\,u(\varepsilon x),\\[1mm]
\dot x
&= - d\,x^2
   + k u x - x(u+\varepsilon x),\\[1mm]
\dot\varepsilon &= 0,
\end{aligned}
\end{equation}
where we again recycle the overdot for the derivative with respect to the rescaled time. Expanding \eqref{eq:preblow_desing_full} for small $\varepsilon$ gives the leading-order dynamics
\begin{equation}\label{eq:preblow_desing_expanded}
\begin{aligned}
\dot u
&= -u\big(\alpha u^2 - \alpha u + d x + u\big)
   + \mathcal{O}(\varepsilon),\\[1mm]
\dot x
&= x\big((k-1)u - d x\big)
   + \mathcal{O}(\varepsilon),\\[1mm]
\dot\varepsilon &= 0.
\end{aligned}
\end{equation}
For $\varepsilon=0$, system \eqref{eq:preblow_desing_expanded} reads as
\begin{equation}\label{eq:preblow_desing_expanded_e0}
\begin{aligned}
\dot u
&= -u\big(\alpha u^2 - \alpha u + d x + u\big)
,\\[1mm]
\dot x
&= x\big((k-1)u - d x\big),
\end{aligned}
\end{equation}
and has the equilibria
\begin{equation}\label{eq:preblow_eps0_eqs}
(u_1,x_1) = (0,0),\qquad
(u_2,x_2)=\Big(1-\tfrac{1}{\alpha},0\Big),\qquad
(u_3,x_3)=\left(1-\tfrac{k}{\alpha},\,\tfrac{(k-1)(\alpha-k)}{\alpha d}\right). 
\end{equation}
Note that the second equilibrium represents the DFE $\mathbf{x_1}$, while the latter represents the blown-up EE $\mathbf{x_2}$ (recall \eqref{eq:expansion1} and \eqref{eq:expansion2}). The only troublesome point is the origin, which is nilpotent; the rest are hyperbolic. In fact, the following is obtained through elementary computations. 

\begin{proposition}
Consider the system \eqref{eq:preblow_desing_expanded_e0} with parameters satisfying $\alpha>1$ and $\beta=d+\varepsilon k$ with $1<k<\alpha$. Then, the equilibrium points
\begin{equation}
(u_1,x_1)=(0,0),\qquad
(u_2,x_2)=\Big(1-\tfrac{1}{\alpha},0\Big),\qquad
(u_3,x_3)=\left(1-\tfrac{k}{\alpha},\,\tfrac{(k-1)(\alpha-k)}{\alpha d}\right)
\end{equation}
satisfy the following spectral properties:
\begin{enumerate}
    \item[(i)] For the equilibrium 
    \begin{equation}
        (u_2,x_2)=\Big(1-\tfrac{1}{\alpha},0\Big),
    \end{equation}
    the corresponding eigenvalues of the Jacobian are
    \begin{equation}
        \lambda_1=2-\alpha-\tfrac{1}{\alpha}<0,\qquad
        \lambda_2=\tfrac{(\alpha-1)(k-1)}{\alpha}>0, 
    \end{equation}
    hence $(u_2,x_2)$ is a hyperbolic saddle. The corresponding eigenvectors can be chosen as 
    \begin{equation}
        \mathbf{v_1}=
        \begin{bmatrix}
        1 \\ 0
        \end{bmatrix}
        , \qquad
        \mathbf{v_2}=
        \begin{bmatrix}
        -\frac{d}{\alpha+k-2} \\ 1
        \end{bmatrix}
        .
    \end{equation}
    \item[(ii)] For the equilibrium 
    \[
    (u_3,x_3)=\left(1-\tfrac{k}{\alpha},\,\tfrac{(k-1)(\alpha-k)}{\alpha d}\right),
    \]
which is the blown-up EE $\mathbf{x_2}$, the corresponding Jacobian reads as
\begin{equation}
    J_3=\begin{bmatrix}
        -\frac{(a-2 k+1) (a-k)}{a} & d \left(\frac{k}{a}-1\right) \\
 \frac{(k-1)^2 (a-k)}{a d} & -\frac{(k-1) (a-k)}{a} 
    \end{bmatrix}.
\end{equation}
    It follows that
    \[
\operatorname{tr}J_3= -\,\frac{(\alpha-k)^2}{\alpha}<0,\qquad
\det J_3=\frac{(\alpha-k)^3(k-1)}{\alpha^2}>0,
\]
hence $(u_3,x_3)$ is locally asymptotically stable (in line with Theorem \ref{teo:EE}).
\end{enumerate}
\end{proposition}

Note that since the equilibria $(u_2,x_2)$ and $(u_3,x_3)$ are both hyperbolic, then \eqref{eq:preblow_desing_full} is a regular perturbation problem in a neighborhood of such equilibria. Therefore, we proceed with blowing-up the origin of \eqref{eq:preblow_desing_full}. For this, we propose the homogeneous blow-up
\begin{equation}
    u= r\bar u,\qquad x= r\bar x,\qquad\ve=r\bar\ve,
\end{equation}
where $\bar u^2+\bar x^2+\bar\ve^2=1$ and $r\geq0$.
The local charts that are relevant to us are given by
\begin{equation}
    \begin{split}
        K_1&: \quad u=r_1u_1,\;x=r_1,\;\ve=r_1\ve_1,\\
        K_2&: \quad u=r_2u_2,\;x=r_2x_2,\;\ve=r_2,\\
        K_3&: \quad u=r_3,\;x=r_3x_3,\;\ve=r_3\ve_3.\\
    \end{split}
\end{equation}
The corresponding transition maps that will be used in the analysis are
\begin{equation}
    \begin{split}
        \kappa_{1\to 2}&:\ (r_1,u_1,\varepsilon_1) \mapsto (r_2,u_2,x_2)=\left(r_1\ve_1,\frac{1}{\ve_1},\frac{u_1}{\ve_1}\right),
\qquad(\varepsilon_1\neq 0),\\
 \kappa_{1\to 3}&:\ (r_1,u_1,\varepsilon_1) \mapsto (r_3,x_3,\varepsilon_3)=\left(r_1u_1,\frac{1}{u_1},\frac{\varepsilon_1}{u_1}\right),
\qquad(u_1\neq 0),\\
\kappa_{2\to 3}&:\ (r_2,u_2,x_2)\mapsto (r_3,x_3,\varepsilon_3)
=\left(r_2u_2,\frac{x_2}{u_2},\frac{1}{u_2}\right),\qquad(u_2\neq 0).
    \end{split}
\end{equation}

\subsubsection*{Chart $K_1$}

In chart $K_1$, the desingularized vector field is:
\begin{equation}\label{eq:BUeK1}
    \begin{split}
        \dot r_1 &= -r_1 \left(d+\ve_1 r_1-(k-1) u_1\right),\\
        \dot u_1 &=-\left(\left(\ve_1 r_1+u_1\right) \left(\alpha  \left(r_1 \left(\ve_1 r_1+u_1\right)-1\right) \left(\ve_1 \theta  r_1+u_1\right)+k u_1\right)\right),\\
        \dot\ve_1&=\ve_1 \left(d+\ve_1 r_1-(k-1) u_1\right).
    \end{split}
\end{equation}
We proceed to study the dynamics in the invariant subsets $\left\{ r_1=\ve_1=0\right\}$, $\left\{\ve_1=0\right\}$, and $\left\{r_1=0\right\}$. Notice that in this chart $\alpha^*=\alpha+\mathcal O(r_1\ve_1)$. Hence, restricted to the aforementioned invariant subsets, $\alpha^*=\alpha$ and hence $1<k<\alpha$.

On $\left\{ r_1=\ve_1=0\right\}$, the dynamics are governed by
\begin{equation}\label{eq:K1e1}
    \dot u_1=u_1^2 (\alpha -k).
\end{equation}
Since $\alpha-k>0$, we have that the flow of \eqref{eq:K1e1} along the $u_1$-axis is always increasing, except at the origin which is an equilibrium.

On $\left\{ \ve_1=0\right\}$, the dynamics are governed by
\begin{equation}\label{eq:K1e2}
    \begin{split}
        \dot r_1 &=-r_1 \left(d-(k-1) u_1\right),\\
        \dot u_1 &= -u_1^2 \left(-\alpha +k+\alpha  r_1 u_1\right).
    \end{split}
\end{equation}

It is readily seen that the origin $(r_1,u_1)=(0,0)$ is a topological saddle with the $r_1$-direction contracting and the $u_1$-direction repelling.
A second equilibrium, namely
\[
(r_1^*,u_1^*)=\left(\frac{(\alpha-k)(k-1)}{\alpha d},\ \frac{d}{k-1}\right),
\]
is located in the interior of the first quadrant as long as $1<k<\alpha$, and corresponds to the endemic equilibrium $\mathbf{x_2}$ (this is not further discussed as it has already been treated before). We notice that the disease-free equilibrium $\mathbf{x_1}$ does not appear in this present chart because it lies on the $u$-axis, corresponding to $\left\{ x=0 \right\}$. Such a set is ``not visible'' in chart $K_1$.

\begin{remark}\leavevmode
\begin{itemize}
    \item In the degenerate case $k=\alpha>1$ we have that \eqref{eq:K1e2} reduces to
    \begin{equation}
    \begin{split}
        \dot r_1 &=-r_1 \left(d-(\alpha-1) u_1\right),\\
        \dot u_1 &= -\alpha  r_1 u_1^3.
    \end{split}
\end{equation}
This system has the entire $u_1$-axis as a set of equilibria, which are repelling for $u_1>\frac{d}{\alpha-1}$ and attracting for $u_1<\frac{d}{\alpha-1}$. Since $u_1$ monotonically decreases, we can see that all orbits converge to the $u_1$-axis asymptotically.
\item In the degenerate case $k=1$ and $\alpha>1$, \eqref{eq:K1e2} reduces to
    \begin{equation}
    \begin{split}
        \dot r_1 &=-dr_1,\\
        \dot u_1 &= -u_1^2(-\alpha+1+\alpha r_1 u_1).
    \end{split}
\end{equation}
In this case, the origin is a topological saddle with the $r_1$-direction contracting and the $u_1$-direction expanding, which is the same as the generic behavior. Note as well that, in this case, the equilibrium $(r_1^*,u_1^*)$ ``leaves the chart''.
\end{itemize}    
\end{remark}

Next, on $\left\{ r_1=0\right\}$, the dynamics are given by:
\begin{equation}\label{eq:K1e3}
    \begin{split}
        \dot u_1 &=u_1^2 (\alpha -k),\\
        \dot\ve_1&=\ve_1 \left(d-(k-1) u_1\right).
    \end{split}
\end{equation}
We have that the flow of \eqref{eq:K1e3} close to the origin is that of a topological source in the first quadrant  (the biologically relevant region). In fact, we can integrate \eqref{eq:K1e3} to get:
\begin{equation}\label{eq:solK1}
\varepsilon_1(u_1)=A u_1^{-(k-1) /(\alpha-k)} \exp \left(-\frac{d}{(\alpha-k) u_1}\right),\qquad 1<k<\alpha, 
\end{equation}
where $A>0$ depends on $(u_1(0),\ve_1(0))$.

We track orbits through chart $K_1$ using the following cross-sections:
\begin{equation}
    \begin{split}
        \Sigma^{\text {in }}&:=\left\{r_1=\rho\right\}, \\
        \Sigma_{1 \rightarrow 3}^{\text {out }}&:=\left\{u_1=1 / \delta_3\right\}, \\
        \Sigma_{1 \rightarrow 2}^{\text {out }}&:=\left\{\varepsilon_1=1 / \delta_2\right\},
    \end{split}
\end{equation}
for suitable small $\rho, \delta_2, \delta_3>0$, intersected with the positive octant and appropriate compact bounds. Since $u_1$ is monotonically increasing on $\left\{r_1=0\right\}$ and $\varepsilon_1 \rightarrow 0$ as $u_1 \rightarrow \infty$ in \eqref{eq:solK1}, every orbit on the sphere eventually exits through $\Sigma_{1 \rightarrow 3}^{\text {out }}$. The function $\varepsilon_1(u_1)$, i.e., \eqref{eq:solK1}, achieves a unique maximum at $\tilde u_1:=\frac{d}{k-1}$, which coincides with the $u_1$-coordinate of the endemic equilibrium on $\left\{\varepsilon_1=0\right\}$.

For $r_1>0$ sufficiently small, the full system \eqref{eq:BUeK1} is a regular perturbation of the sphere dynamics. By continuous dependence on initial conditions and the monotonic growth of $u_1$, every orbit entering through $\Sigma^{\text {in }}$ eventually exits through $\Sigma_{1 \rightarrow 3}^{\text {out }}$, possibly after a transient excursion through chart $K_2$ (see below) during the phase when $\varepsilon_1$ is increasing $\left(u_1<\tilde u_1\right)$. The latter excursion corresponds, in the original system \eqref{eq:preblow_desing_full}, to the orbit's approach to $\mathbf{x_0}$ along the $x$-axis before being repelled along the $u$-axis. Details of the dynamics in charts $K_2$ and $K_3$ are provided below.

\begin{remark}
    Due to \eqref{eq:solK1}, it may seem reasonable to argue that passing to the chart $K_2$ is not necessary. To guarantee this, we would need to show that $\ve_1$ remains bounded for all $r_1>0$. This is true, and it is straightforward to verify in chart $K_2$.
\end{remark}

Figure \ref{fig:BUeK1} shows a qualitative picture of the dynamics in chart $K_1$.
\begin{figure}[htbp]
    \centering
    \begin{tikzpicture}
        \node at (0,0){\includegraphics[]{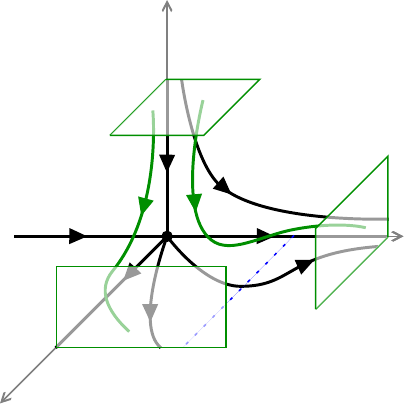}};
        \node at (3.5,-0.35){$u_1$};
        \node at (-0.9,3.35){$r_1$};
        \node at (-3.5,-3){$\ve_1$};
        \node at (1,2.35){$\Sigma^{\text {in }}$};
        \node at (-1.75,-2.8){$\Sigma_{1 \rightarrow 2}^{\text {out }}$};
        \node at (2.5,-2){$\Sigma_{1 \rightarrow 3}^{\text {out }}$};
    \end{tikzpicture}
    \caption{Qualitative dynamics of \eqref{eq:BUeK1} close to the origin. This picture qualitatively describes the dynamics of \eqref{eq:preblow_desing_full} close to the origin and with $x>0$. The black orbits are on the invariant planes, while the green ones evolve on the full $3$-dimensional space. The dashed blue line indicates $\tilde u_1$, that is, where $\dot \ve_1$ changes sign. The three sections $\Sigma^{\text {in }},\; 
        \Sigma_{1 \rightarrow 3}^{\text {out }},\; 
        \Sigma_{1 \rightarrow 2}^{\text {out }}$ are also indicated. 
        }
    \label{fig:BUeK1}
\end{figure}
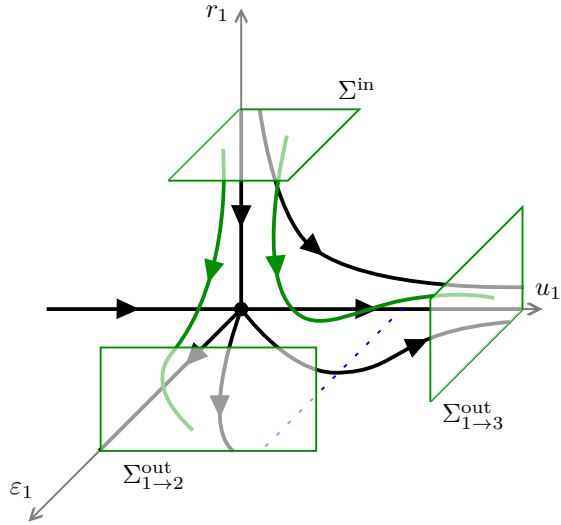

\begin{remark}\leavevmode
\begin{itemize}
    \item In the degenerate case $k=\alpha>1$, \eqref{eq:K1e3} reduces to 
    \begin{equation}
        \begin{split}
            \dot u_1&=0\\
            \dot\ve_1&=(d-(\alpha-1)u_1)
        \end{split}
    \end{equation}
    The flow of this system consists of straight lines parallel to the $\ve_1$-axis, i.e., with $u_1$ fixed. The flow goes away from the $u_1$-axis for $u_1<\frac{d}{\alpha-1}$ and approaches the $u_1$-axis for $u_1>\frac{d}{\alpha-1}$.
    \item In the degenerate case $k=1$ and $\alpha>1$,  \eqref{eq:K1e3} reduces to  
    \begin{equation}
    \begin{split}
        \dot u_1 &=u_1^2(\alpha-1),\\
        \dot\ve_1&=d\ve_1,
    \end{split}
\end{equation}
whose qualitative behavior near the origin is the same as in the generic case.
\end{itemize}
    
\end{remark}

\subsubsection*{Chart $K_2$}

In chart $K_2$, the desingularized vector field has the usual $\dot r_2=0$ and restricted to $\left\{r_2=0\right\}$ reads as:
\begin{equation}\label{eq:BUeK2}
    \begin{split}
        \dot u_2 &=u_2 \left((\alpha -1) u_2-d x_2\right),\\
        \dot x_2 &=x_2 \left((k-1) u_2-d x_2\right).
    \end{split}
\end{equation}
Analogous to chart $K_1$, in $K_2$ we have that $\alpha^*=\alpha+\mathcal O(r_2)$. Hence, restricting to $r_2=0$ leads to $\alpha^*=\alpha$.

Since $k<\alpha$, the origin is the unique equilibrium point of \eqref{eq:BUeK2}. In addition, although the origin is nilpotent, we can readily see that the $u_2$-axis is invariant and the flow on the axis is directed away from the origin. Similarly, the $x_2$-axis is invariant and the flow on the axis is directed towards the origin. Hence, in the first quadrant, the origin is a topological saddle with the $x_2$-direction attracting and the $u_2$-direction repelling. The lines
\begin{equation}
    \begin{split}
        c_1&=\left\{x_2=\frac{\alpha-1}{d}u_2\right\},\\
        c_2&=\left\{x_2=\frac{k-1}{d}u_2\right\},
    \end{split}
\end{equation}
also play a descriptive role: orbits above $c_1$ decrease in both coordinates, while those below $c_2$ increase in both coordinates. 

To track the flow in this chart we define
\begin{equation}
    \begin{split}
        \Sigma_2^{\text {in }}&:=\left\{x_2=\delta_2\right\}, \\
        \Sigma_2^{\text {out }}&:=\left\{u_2=1 / \delta_4\right\},
    \end{split}
\end{equation}
for suitable small $\delta_2, \delta_4>0$, intersected with the positive octant and appropriate compact bounds. Notice that $ \Sigma_2^{\text {in }}=\kappa_{1\to2}(\Sigma_{1 \rightarrow 2}^{\text {out }})$. The fact that the flow is that of a topological saddle guarantees that orbits starting at $\Sigma_2^{\text {in }}$ reach $\Sigma_2^{\text {out }}$ in finite time. In summary, the flow in this chart is as depicted in Figure \ref{fig:BUeK2}.
\begin{figure}[htbp]
    \centering
    \begin{tikzpicture}
        \node at (0,0){\includegraphics[]{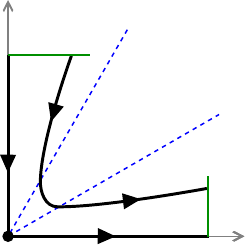}};
        \node at (2,-2.25){$u_2$};
        \node at (-2.3,2){$x_2$};
        \node at (-1,1.5){$\Sigma_2^{\text {in }}$};
        \node at (2,-1){$\Sigma_2^{\text {out }}$};
        \node at (0.25,1.75){\textcolor{blue}{$c_1$}};
        \node at (1.85,0.25){\textcolor{blue}{$c_2$}};
    \end{tikzpicture}
    \caption{Qualitative dynamics of \eqref{eq:BUeK2} close to the origin, which describe the dynamics of \eqref{eq:preblow_desing_full} with $\ve\geq0$ sufficiently small and within an $\mathcal O(\ve)$ neighborhood of the origin. The $c_1$, $c_2$ lines are indicated in dashed blue. We remark that since $\dot r_2=0$, the chart $K_2$ is foliated by planes where the flow is the same as the one indicated in the picture. }
    \label{fig:BUeK2}
\end{figure}
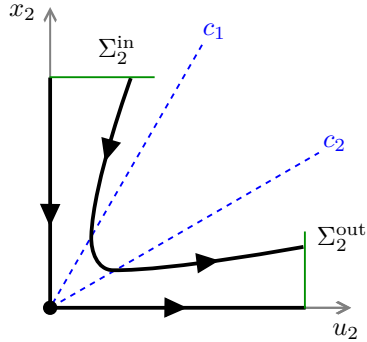

\begin{remark}\leavevmode
\begin{itemize}
    \item In the degenerate case $k=\alpha>1$, \eqref{eq:BUeK2} reduces to
    \begin{equation}
    \begin{split}
        \dot u_2 &=u_2 \left((\alpha -1) u_2-d x_2\right),\\
        \dot x_2 &=x_2 \left((\alpha-1) u_2-d x_2\right).
    \end{split}
\end{equation}
This system has a line of equilibria given by $\ell_2=\left\{x_2=\frac{(\alpha-1)u_2}{d}\right\}$, while the axes still have the flow as described above. Clearly, due to the common factor $\left((\alpha-1) u_2-d x_2\right)$, orbits above $\ell_2$ approach the origin while orbits below $\ell_2$ are unbounded.
\item In the degenerate case $k=1$ and $\alpha>1$, \eqref{eq:BUeK2} reduces to
\begin{equation}
    \begin{split}
        \dot u_2 &=u_2 \left((\alpha -1) u_2-d x_2\right),\\
        \dot x_2 &=-dx_2^2,
    \end{split}
\end{equation}
whose qualitative behavior near the origin is the same as in the generic case.
\end{itemize}
\end{remark}

\subsubsection*{Chart $K_3$}

Finally, we have that the desingularized vector field in the chart $K_3$ is given by:
\begin{equation}\label{eq:BUeK3}
    \begin{split}
        \dot r_3 &=-r_3(1-\alpha +d x_3+\mathcal O(r_3)),\\
        \dot x_3 &=x_3(1+\ve_3 r_3 x_3)(k-\alpha+\mathcal O(r_3)),\\
        \dot\ve_3&=\ve_3(1-\alpha +d x_3+\mathcal O(r_3)).
    \end{split}
\end{equation}
Just as for chart $K_1$, we are going to consider the invariant sets $\left\{ r_3=\ve_3=0\right\}$, $\left\{ r_3=0\right\}$, and $\left\{ \ve_3=0\right\}$. 
    As in the previous charts, in this one we also have that $\alpha^*=\alpha$ when restricting to the aforementioned invariant sets.

On $\left\{ r_3=\ve_3=0\right\}$, the dynamics are given by
\begin{equation}\label{eq:BUeK3_1}
    \dot x_3=(k-\alpha)x_3.
\end{equation}
Since $k-\alpha<0$, the origin of \eqref{eq:BUeK3_1} is a hyperbolic sink. 

Next, system \eqref{eq:BUeK3} on $\left\{ r_3=0\right\}$ reads as
\begin{equation}\label{eq:BUeK3_2}
    \begin{split}
        \dot x_3&=(k-\alpha)x_3,\\
        \dot\ve_3&=\ve_3(1-\alpha+dx_3).
    \end{split}
\end{equation}
Since $1<k<\alpha$, we have that the origin of \eqref{eq:BUeK3_2} is (also) a hyperbolic sink. 
\begin{remark}\leavevmode
\begin{itemize}
    \item In the degenerate case $k=\alpha>1$, \eqref{eq:BUeK3_2} reduces to
    \begin{equation}
    \begin{split}
        \dot x_3&=0,\\
        \dot\ve_3&=\ve_3(1-\alpha+dx_3).
    \end{split}
\end{equation}
The flow is given by straight lines parallel to the $\ve_3$-axis. The flow approaches the $x_3$-axis for $x_3<\frac{\alpha-1}{d}$ and is repelled from it for $x_3>\frac{\alpha-1}{d}$.
\item In the degenerate case $k=1$ and $\alpha>1$, \eqref{eq:BUeK3_2} reduces to
\begin{equation}
    \begin{split}
        \dot x_3&=(1-\alpha)x_3,\\
        \dot\ve_3&=\ve_3(1-\alpha+dx_3),
    \end{split}
\end{equation}
whose qualitative behavior near the origin is the same as in the generic case.
\end{itemize}
    
\end{remark}

Similarly, on $\left\{ \ve_3=0\right\}$ the dynamics read as
\begin{equation}\label{eq:BUeK3_3}
    \begin{split}
        \dot r_3 &= -r_3 \left(-\alpha +d x_3+\alpha  r_3+1\right),\\
        \dot x_3 &= x_3 \left(-\alpha +k+\alpha  r_3\right).
    \end{split}
\end{equation}

In the first quadrant and for $1<k<\alpha$, \eqref{eq:BUeK3_3} has exactly three equilibria,
\[
E_0=(r_3^{(0)},x_3^{(0)})=(0,0),\quad 
E_1=(r_3^{(1)},x_3^{(1)})=\left(\frac{\alpha-1}{\alpha},0\right),\quad
E_2=(r_3^{(2)},x_3^{(2)})=\left(\frac{\alpha-k}{\alpha},\ \frac{k-1}{d}\right).
\]

Here, $E_1$ corresponds to the disease-free equilibrium $\mathbf{x_1}$, while $E_2$ corresponds to the endemic equilibrium. 

It follows from standard linearization arguments that the origin of \eqref{eq:BUeK3_3} is a hyperbolic saddle with the $r_3$-direction being repelling and the $x_3$-direction attracting. 
To track the orbits in $K_3$ coming from $K_1$ and $K_2$ we define the sections:
\begin{equation}
    \begin{split}
        \Sigma_{1 \rightarrow 3}^{\text {in }}&:=\left\{x_3=\delta_3\right\}\\
        \Sigma_{2 \rightarrow 3}^{\text {in }}&:=\left\{\ve_3=\delta_4\right\}\\
        \Sigma_{3}^{\text {out }}&:=\left\{r_3=\rho_3\right\},
    \end{split}
\end{equation}
such that $\Sigma_{1 \rightarrow 3}^{\text {in }}=\kappa_{1\to3}(\Sigma_{1 \rightarrow 3}^{\text {out}})$ and $\Sigma_{2 \rightarrow 3}^{\text {in }}=\kappa_{2\to3}(\Sigma_{2}^{\text {out }})$. Since these section are transverse to the flow, and we have already argued that the origin is a hyperbolic saddle, it follows that all orbits entering $K_3$ reach $\Sigma_{3}^{\text {out }}$ in finite time. A qualitative picture of the flow in chart $K_3$ is given in Figure \ref{fig:BUeK3}.
\begin{figure}[htbp]
    \centering
    \begin{tikzpicture}
        \node at (0,0){\includegraphics[]{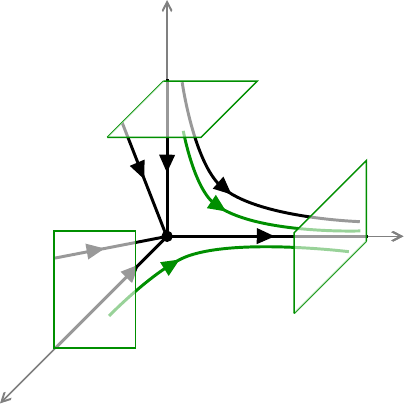}};
        \node at (3.5,-0.35){$r_3$};
        \node at (-0.9,3.35){$x_3$};
        \node at (-3.5,-3){$\ve_3$};
        \node at (-3,-2){$\Sigma_{2 \rightarrow 3}^{\text {in }}$};
        \node at (0.75,2.35){$\Sigma_{1 \rightarrow 3}^{\text {in }}$};
        \node at (3,0.9){$\Sigma_{3}^{\text {out }}$};
    \end{tikzpicture}
    \caption{Qualitative dynamics of \eqref{eq:BUeK3} for $1<k<\alpha$ close to the origin. This picture describes the dynamics of \eqref{eq:preblow_desing_full} close to the origin and with $u>0$. The black orbits are on the invariant planes, while the green ones evolve on the full $3$-dimensional space. The disease-free and the endemic equilibria are omitted.
    \label{fig:BUeK3}}
\end{figure}
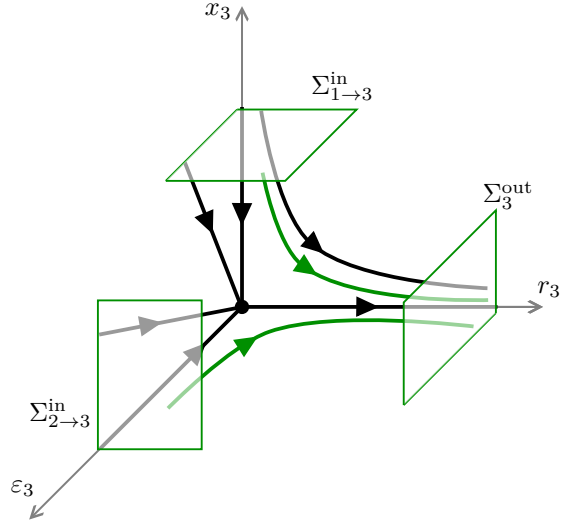

\begin{remark}\leavevmode
\begin{itemize}
    \item In the degenerate scenario $k=\alpha>1$, \eqref{eq:BUeK3_3} reduces to
    \begin{equation}
    \begin{split}
        \dot r_3 &= -r_3 \left(-\alpha +d x_3+\alpha  r_3+1\right),\\
        \dot x_3 &= \alpha x_3 r_3.
    \end{split}
\end{equation}
This system has the whole $x_3$-axis as a line of equilibria. These equilibria are repelling for $x_3<\frac{\alpha-1}{d}$ and attracting for $x_3>\frac{\alpha-1}{d}$. On the other hand, the disease-free equilibrium $E_1$ persists. Notice as well that as $k\to\alpha$, the endemic equilibrium collapses into the $x_3$-axis precisely at $x_3=\frac{\alpha-1}{d}$. Finally, the orbits (in the interior of the first quadrant) are attracted towards the invariant line $\left\{ -\alpha +d x_3+\alpha  r_3+1=0\right\}$, and the flow is unbounded in the positive $x_3$-direction.
\item In the degenerate scenario $k=1$ and $\alpha>1$, \eqref{eq:BUeK3_3} reduces to
\begin{equation}
    \begin{split}
        \dot r_3 &= -r_3 \left(-\alpha +d x_3+\alpha  r_3+1\right),\\
        \dot x_3 &= x_3 \left(-\alpha +1+\alpha  r_3\right),
    \end{split}
\end{equation}
whose qualitative behavior near the origin is the same as in the generic case.
\end{itemize}

\end{remark}

Through the above analysis we have shown that the generic dynamics (i.e., for $1<k<\alpha^*$) starting at $\Sigma_1^{\text{in}}$ eventually reach $\Sigma_3^{\text{out}}$. Via the blow-down we can conclude that the origin of \eqref{eq:preblow_desing_full} is a topological saddle, which proves the last item of Theorem \ref{thm:blow}. Figure \ref{fig:BUe} shows a qualitative picture of the global blow-up dynamics and the corresponding blow-down. 

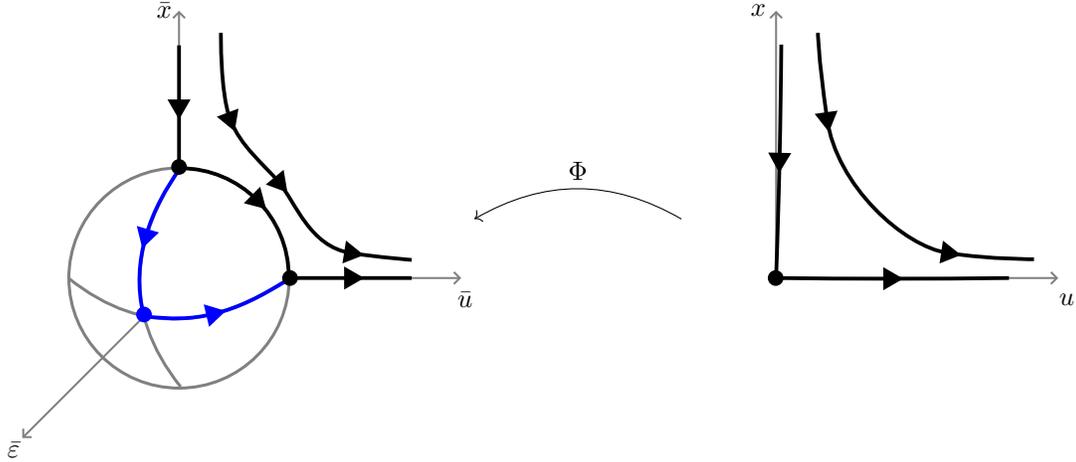
\begin{figure}[htbp]
    \centering
    \begin{tikzpicture}
        \node at (0,0){\includegraphics[]{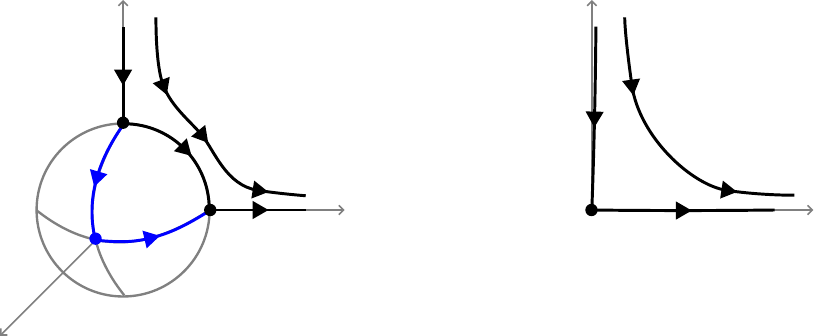}};
        \node at (-5,2.85){$\bar x$};
        \node at (-1.0,-1){$\bar u$};
        \node at (2.9,2.85){$x$};
        \node at (7,-1){$u$};
        \node at (-7,-3){$\bar\ve$};

        \node (A) at (2,0) {};
        \node (B) at (-1,0) {};
        \draw[->, bend right=30] (A) to node[midway, above]{$\Phi$} (B);
    \end{tikzpicture}
     \caption{For $\alpha>1$, $\beta=d+\varepsilon k$ with $1<k<\alpha$, and $\varepsilon>0$ sufficiently small: (left) behavior near and on the blown-up sphere corresponding to the origin of \eqref{eq:preblow_desing_full}; (right) behavior of \eqref{eq:preblow_desing_full} near the origin for $\ve>0$ sufficiently small. We note that the overall picture follows from the blow-up, however, the $x$-axis is not invariant for \eqref{eq:preblow_desing_full} if $\theta \ne 0$. A simple computation shows that $\dot u|_{u=0}$ is nonnegative and of order $\mathcal O(\theta x\ve^2)$. \label{fig:BUe}}
\end{figure}

Although we omit the formal details, due to high degeneracy, through the several remarks above, we can also observe that:
\begin{itemize}
    \item As $k\to1$ (and with $\alpha>1$), the origin has the same behavior as for the generic case, that is, the origin is a topological saddle.
    \item As $k\to\alpha^*$ (and with $\alpha>1$) the origin transitions towards the $\beta-d\in\mathcal O(1)$ regime. The corresponding blown-up picture is as in Figure \ref{fig:BUdegenerate}.
    
    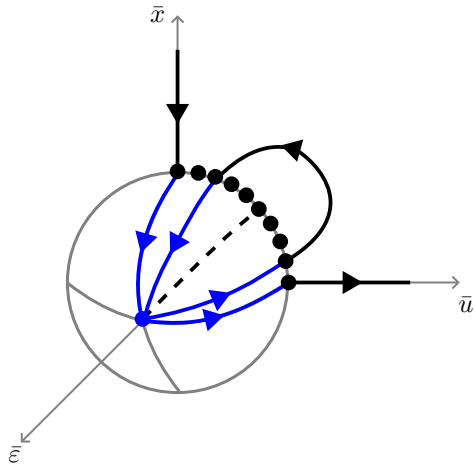
\begin{figure}[htbp]
        \centering
        \begin{tikzpicture}
            \node at (0,0){\includegraphics[]{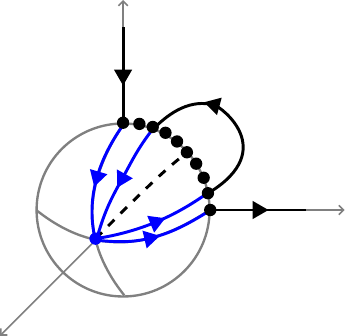}
            };
            \node at (-1.1,2.85){$\bar x$};
        \node at (3,-1){$\bar u$};
        \node at (-3,-3){$\bar\ve$};
        \end{tikzpicture}
        \caption{In the blow-up space, the endemic equilibrium in chart $K_3|_{\ve_3=0}$ collides with the blow-up sphere as $k\to\alpha$. This corresponds to $k\to\alpha^*$ for $\varepsilon>0$ sufficiently small in \eqref{eq:preblow_desing_full}. In this scenario, the equator of the sphere degenerates to a curve of equilibria and the curves $c_1$,$c_2$ in chart $K_2$ collide (indicated by the dashed black curve). This marks the transition between the third and the second item in Proposition \ref{thm:blow}.}
        \label{fig:BUdegenerate}
    \end{figure}
\end{itemize}

\bigskip
\bigskip
\noindent\textbf{Acknowledgments.} JB and MS are members and acknowledge the support of {\it Gruppo Nazionale di Fisica Matematica} (GNFM) of {\it Istituto Nazionale di Alta Matematica} (INdAM). JB acknowledges the support of the project PRIN 2022 PNRR ``Mathematical Modelling for a Sustainable Circular Economy in Ecosystems'' (project code P2022PSMT7, CUP D53D23018960001) funded by the European Union - NextGenerationEU, PNRR-M4C2-I 1.1, and by MUR-Italian Ministry of Universities and Research.

{\footnotesize
	\bibliographystyle{unsrt}
	\bibliography{biblio}
}

\end{document}